\documentclass[12pt]{article}
\usepackage[margin=1in]{geometry}
\usepackage{amsthm, amsmath,amsfonts,amssymb,euscript,hyperref,graphics,color,slashed,mathrsfs}
\usepackage{graphicx}
\usepackage{comment}
\usepackage{import}
\usepackage{tikz}
\usepackage{latexsym}
\usepackage{mathtools}
\usepackage{appendix}
\usepackage{stmaryrd}
\usetikzlibrary{arrows}
\usepackage{pgfplots}
\usepackage{extarrows}

\allowdisplaybreaks

\makeatletter
\newcommand{\subjclass}[2][1991]{%
	\let\@oldtitle\@title%
	\gdef\@title{\@oldtitle\footnotetext{#1 \emph{Mathematics subject classification.} #2}}%
}
\newcommand{\keywords}[1]{%
	\let\@@oldtitle\@title%
	\gdef\@title{\@@oldtitle\footnotetext{\emph{Key words and phrases.} #1.}}%
}
\makeatother
\newtheorem{theorem}{Theorem}[section]
\newtheorem*{theorem*}{Theorem}
\newtheorem{lemma}[theorem]{Lemma}

\newtheorem{corollary}[theorem]{Corollary}

\theoremstyle{remark}
\newtheorem{definition}[theorem]{\bf Definition}

\newtheorem{conjecture}[theorem]{\bf Conjecture}

\newcommand{\de}{\mathrm{d}}
\numberwithin{equation}{section}
\numberwithin{figure}{section}

\begin{document}
\title {A partial result towards the Chowla--Milnor conjecture}
	
\author{Li Lai, Jia Li}
\date{}
\subjclass[2020]{11J72 (primary), 11M35, 33C20 (secondary).}
\keywords{Linear independence, Hurwitz zeta functions, hypergeometric series, the Chowla--Milnor conjecture}
	
\maketitle
	
\begin{abstract}
The Chowla--Milnor conjecture predicts the linear independence of certain Hurwitz zeta values. 
In this paper, we prove that for any fixed integer $k \geqslant 2$, the dimension of the $\mathbb{Q}$-linear span of $\zeta(k,a/q)-(-1)^{k}\zeta(k,1-a/q)$ ($1 \leqslant a < q/2$, $\gcd(a,q)=1$) is at least $(c -o(1)) \cdot \log q$ as the positive integer $q \to +\infty$ for some absolute constant $c>0$.  
It is well known that $\zeta(k,a/q)+(-1)^{k}\zeta(k,1-a/q) \in \overline{\mathbb{Q}}\pi^k$, but much less is known about $\zeta(k,a/q)-(-1)^{k}\zeta(k,1-a/q)$.
Our proof is similar to those of Ball--Rivoal (2001) and Zudilin (2002) concerning the linear independence of Riemann zeta values. 
However, we use a new family of rational functions to construct linear forms.
\end{abstract}

%\tableofcontents

\section{Introduction}
For a real number $x$ with $0 < x \leqslant 1$, the Hurwitz zeta function is defined by
\[ 
\zeta(s,x) := \sum_{m=0}^{+\infty} \frac{1}{(m+x)^s}, \quad \operatorname{Re}(s) > 1. \]
In the special case $x=1$, the Hurwitz zeta function $\zeta(s,1)$ reduces to the Riemann zeta function $\zeta(s)$. We also recall that
\begin{equation}
\label{eq:half-argument}
\zeta\left(k,\frac12\right)
=
2^k\sum_{m=0}^{+\infty}\frac{1}{(2m+1)^k}
=
(2^k-1)\zeta(k)
\qquad (k\geqslant2).
\end{equation}
We therefore focus on $q\geqslant3$.
We are interested in the arithmetic nature of special values of Hurwitz zeta functions. 
According to \cite{GMR2011}, it was conjectured by Chowla and Chowla \cite{CC1982} that, for any prime number $p$, the $p-1$ Hurwitz zeta values $\zeta(2,1/p),\zeta(2,2/p),\ldots,\zeta(2,(p-1)/p)$ are linearly independent over $\mathbb{Q}$. 
Their conjecture was generalized by Milnor \cite{Mil1983} as follows, now known as the Chowla--Milnor conjecture.
	
\begin{conjecture}[The Chowla--Milnor conjecture, 1983]
Let $k \geqslant 2$ and $q \geqslant 3$ be integers. 
Then the following $\varphi(q)$ Hurwitz zeta values are linearly independent over $\mathbb{Q}$:
\[ 
\zeta\left( k, \frac{a}{q} \right),   \quad 1 \leqslant a < q \text{~with~} \gcd(a,q)=1.
\]
\end{conjecture}
	
A recent breakthrough by Calegari, Dimitrov, and Tang \cite{CDT2024+} confirms two special cases of the Chowla--Milnor conjecture:
	
\begin{theorem}[Calegari--Dimitrov--Tang \cite{CDT2024+}, 2024+]\label{thm_CDT}
The three numbers $1$, $\zeta(2)$, and $L(2,\chi_{-3})$ are linearly independent over $\mathbb{Q}$, where $L(2,\chi_{-3}) = (\zeta(2,1/3)-\zeta(2,2/3))/9$.
In particular, the Chowla--Milnor conjecture is true in the special cases $(k,q)=(2,3)$ and $(k,q)=(2,6)$.
\end{theorem}
	
The last sentence of Theorem \ref{thm_CDT} follows from 
\[
\begin{pmatrix}
\zeta\!\left(2,\frac13\right) \\[2pt]
\zeta\!\left(2,\frac23\right)
\end{pmatrix}
=
\begin{pmatrix}
4 &  \frac{9}{2} \\
4 & -\frac{9}{2}
\end{pmatrix}
\begin{pmatrix}
\zeta(2) \\[2pt]
L(2,\chi_{-3})
\end{pmatrix} \quad\text{and}\quad \begin{pmatrix}
\zeta\!\left(2,\frac16\right) \\[2pt]
\zeta\!\left(2,\frac56\right)
\end{pmatrix}
=
\begin{pmatrix}
12 &  \frac{45}{2} \\
12 & -\frac{45}{2}
\end{pmatrix}
\begin{pmatrix}
\zeta(2) \\[2pt]
L(2,\chi_{-3})
\end{pmatrix}.
\]
To our knowledge, all other cases of the Chowla--Milnor conjecture remain open. 

Following the terminology of Gun, Murty, and Rath in \cite{GMR2011}, we define the Chowla--Milnor space $V_k(q)$ as follows.
	
\begin{definition}
\label{def_V_k(q)}
Let $k \geqslant 2$ and $q \geqslant 3$ be integers. 
For any integer $a \in \{1,2,\ldots,q-1\}$, we define the even part $\zeta^{+}(k,a/q)$ and the odd part $\zeta^{-}(k,a/q)$ of the Hurwitz zeta value $\zeta(k,a/q)$ by
\begin{align*}
\zeta^{+}\left(k,\frac{a}{q}\right)&:=\zeta\left( k, \frac{a}{q} \right) + (-1)^{k}\zeta\left( k, 1- \frac{a}{q} \right),\\
\zeta^{-}\left(k,\frac{a}{q}\right)&:=\zeta\left( k, \frac{a}{q} \right) - (-1)^{k}\zeta\left( k, 1- \frac{a}{q} \right).
\end{align*}
We define the following three $\mathbb{Q}$-linear spaces:
\begin{align*}
V_k(q) &:= \operatorname{Span}_{\mathbb{Q}}\left\{  \zeta\left( k, \frac{a}{q} \right) ~~\Big|~~ 1 \leqslant a < q,~\gcd(a,q)=1 \right\}, \\
V_k^{+}(q) &:= \operatorname{Span}_{\mathbb{Q}}\left\{  \zeta^{+}\left(k,\frac{a}{q}\right)  ~~\Big|~~ 1 \leqslant a < \frac{q}{2},~\gcd(a,q)=1 \right\},  \\
V_k^{-}(q) &:= \operatorname{Span}_{\mathbb{Q}}\left\{  \zeta^{-}\left(k,\frac{a}{q}\right)  ~~\Big|~~ 1 \leqslant a < \frac{q}{2},~\gcd(a,q)=1 \right\}.
\end{align*}
\end{definition}
	
\bigskip
	
Clearly, for any integers $k \geqslant 2$ and $q \geqslant 3$ we have
\[ 
V_k(q) = V_k^{+}(q) + V_k^{-}(q). 
\]
The Chowla--Milnor conjecture can be formulated as $\dim_{\mathbb{Q}} V_k(q) = \varphi(q)$. 
It is well known (see \cite[\S 1]{GMR2011}) that
\begin{equation}
\label{even_part}
V_k^{+}(q) \subset (2\pi i )^k \mathbb{Q}(e^{2\pi i /q}), \quad \dim_{\mathbb{Q}} V_k^{+}(q) = \frac{\varphi(q)}{2}. 
\end{equation}
Therefore, the `even subspace' $V^{+}_k(q)$ is well understood. 
In contrast, the `odd subspace' $V^{-}_{k}(q)$ is more mysterious. 
For example, if $q=4$, then
\[ 
V^{-}_{k}(4) = 
\begin{cases}
\zeta(k)\mathbb{Q} \quad&\text{if $k \geqslant 3$ is odd,} \\
\beta(k)\mathbb{Q} \quad&\text{if $k \geqslant 2$ is even,}
\end{cases} 
\]
where $\beta(\cdot)$ denotes the Dirichlet beta function. 
We know little about the arithmetic nature of $\zeta(k)$ (for odd $k \geqslant 3$) and $\beta(k)$ (for even $k \geqslant 2$).  
By \eqref{even_part}, the Chowla--Milnor conjecture is equivalent to the following.
	
\begin{conjecture}[An equivalent form of the Chowla--Milnor conjecture]
\label{conj}
Let $k \geqslant 2$ and $q \geqslant 3$ be integers. 
Then we have
\begin{enumerate}
\item[\textup{(1)}]  $\dim_{\mathbb{Q}}  V_k^{-}(q) = \varphi(q)/2$,
\item[\textup{(2)}]  $ V_k^{+}(q) \cap V_k^{-}(q) = \{ 0 \}$.  
\end{enumerate}
\end{conjecture}
Part (2) of Conjecture \ref{conj} seems out of reach. 
The purpose of this paper is to provide partial evidence for part (1) of Conjecture \ref{conj}. 
Our main result is as follows:
	
\begin{theorem}
\label{main}
Fix any integer $k \geqslant 2$. 
Then, as the positive integer $q \to +\infty$, we have
\[ 
\dim_{\mathbb{Q}} V_k^{-}(q) \geqslant \left(\frac{1}{\log 2} - o(1)\right) \cdot \log q. 
\]
\end{theorem}
	
Our proof of Theorem \ref{main} is similar to those of Ball--Rivoal \cite{BR2001} and Zudilin \cite{Zud2002} regarding the linear independence of Riemann zeta values. 
The novelty of our paper lies in a new family of rational functions. 
Using these rational functions, we construct linear forms in $1$ and certain elements of $V_k^{-}(q)$. 
We then apply a version of Nesterenko's linear independence criterion \cite{Fis2012,FZ2010,Nes1985} to obtain a lower bound for $\dim_{\mathbb{Q}} V_k^{-}(q)$. 
To estimate these linear forms, we employ the saddle-point method.
We also mention that in \cite{Fis2020}, Fischler obtained other results related to linear independence of Hurwitz zeta values (and Dirichlet $L$-values).
	
\bigskip
	
The structure of this paper is as follows. 
In \S \ref{Sect_Prel}, we introduce Nesterenko's linear independence criterion and the saddle-point method.
In \S \ref{Sect_RL}, we first construct rational functions $R_n(t)$ and linear forms $S_n$. 
Then, we study the coefficients of these linear forms. 
In \S \ref{Sect_Prop_of_Hurw}, we prove a property of the Hurwitz zeta values.
In \S \ref{Sect_Int_expre_of_S_n}, we express $S_n$ as complex integrals.
In \S \ref{Sect_h(z)}, we carefully analyze a class of functions to locate saddle points.
In \S \ref{Sect_asym_S_n}, we use the saddle-point method to obtain asymptotic estimates of $S_n$ as $n \to +\infty$.
Finally, we prove Theorem \ref{main} in \S \ref{Sect_proof}.
	
\textbf{Notation}:  Throughout this paper, the function $\log(\cdot)$
denotes the principal branch of the logarithm function on the cut plane $\mathbb{C} \setminus (-\infty,0]$. 
The notation $\mathbb{N}$ denotes the set of positive integers, and $i$ is used to represent the imaginary unit.
	
\section{Preliminaries}
\label{Sect_Prel}
	
In this section, we introduce Nesterenko's linear independence criterion and the saddle-point method. 
These are the basic tools for our proof of Theorem \ref{main}.
	
In 1985, Nesterenko \cite{Nes1985} established a linear independence criterion similar to Siegel's classical criterion \cite{Sie1929}. 
Nesterenko's criterion proves to be useful in many situations, including the context of the Ball--Rivoal theorem \cite{BR2001}. 
For our purposes, we need the following refined version of Nesterenko's criterion, which addresses oscillating linear forms.
	
\begin{theorem}
\label{thm_Nes_osc}
Let $m \in \mathbb{N}$ and $\xi_1,\xi_2,\ldots,\xi_m$ be real numbers. 
Let $\alpha$ and $\beta$ be positive constants. 
Let $\omega$ and $\varphi$ be real constants such that 
\[ 
\omega \notin \pi \mathbb{Z} \quad\text{or}\quad \varphi \notin \frac{\pi}{2} +\pi\mathbb{Z}.  
\]
For any $n \in \mathbb{N}$, let 
\[ 
L_n(X_0,X_1,\ldots,X_m) = \sum_{j=0}^{m} l_{n,j}X_j 
\]
be a linear form in $m+1$ variables with integer coefficients $l_{n,j} \in \mathbb{Z}$ \textup{(}$j=0,1,\ldots,m$\textup{)}. 
Suppose that the following conditions hold:
\begin{itemize}
\item $|L_n(1,\xi_1,\xi_2,\ldots,\xi_m)| = \exp\left( -\alpha n + o(n)  \right) \cdot (|\cos(n\omega+\varphi)|+o(1))$ as $n \to +\infty$\textup{;}
\item $\max\limits_{0\leqslant j \leqslant m} |l_{n,j}| \leqslant \exp\left( \beta n + o(n) \right)$ as $n \to +\infty$.
\end{itemize}
Suppose, in addition, that for each $n \in \mathbb{N}$ and each $j=1,2,\ldots,m$, the integer coefficient $l_{n,j}$ has a positive divisor $d_{n,j}$ such that\textup{:}
\begin{itemize}
\item $d_{n,j}$ divides $d_{n,j+1}$ for any $n \in \mathbb{N}$ and any $j=1,2,\ldots,m-1$\textup{;}
\item $d_{n,k}/d_{n,j}$ divides $d_{n+1,k}/d_{n+1,j}$ for any $n \in \mathbb{N}$ and any $0 \leqslant j < k \leqslant m$, with $d_{n,0}=1$\textup{;}
\item For each $j=1,2,\ldots,m$, we have $d_{n,j}= \exp\left( \gamma_j n + o(n) \right)$ as $n \to +\infty$ for some constant $\gamma_j \geqslant 0$.
\end{itemize}
Let $d=\dim_{\mathbb{Q}}\operatorname{Span}_{\mathbb{Q}}\left( 1,\xi_1,\xi_2,\ldots,\xi_m \right)$.
Then, we have 
\[ 
d \geqslant 1 +\frac{\alpha+\gamma_1+\gamma_2+\cdots+\gamma_{d-1}}{\beta}. 
\]
\end{theorem}   

\begin{proof}
By \cite[Proposition 1]{Fis2012}, there exist $\varepsilon,\lambda>0$ and an increasing function $\psi:\mathbb{N}\to\mathbb{N}$ such that $\lim\limits_{n\to+\infty}\psi(n)/n=\lambda$, and for any $n\in\mathbb{N}$, $|\cos(\psi(n)\omega + \varphi)|\geqslant\varepsilon$.
Then we apply \cite[Theorem 1]{FZ2010} to the subsequence $\{L_{\psi(n)}(X_0,X_1,\cdots,X_m)\}$ with $Q_{n} = \exp({\beta\lambda n})$ to obtain the claimed inequality of this theorem.

\end{proof}

In our proof of Theorem \ref{main}, we will use the saddle-point method to estimate certain linear forms. 
Our situation is analogous to that of Zudilin \cite{Zud2002}.
The saddle-point method is classical, but there are many different versions in the literature; see, for example, \cite[Theorem 4, p. 105]{Won2001}.
For the reader's convenience, we present below an elementary version of the saddle-point method, which is sufficient for our purposes.

\begin{theorem}[The saddle-point method]
\label{saddle}
Let $f$, $g$, and $g_n$ $(n\geqslant 1)$ be holomorphic
functions on a domain $\mathcal{D}\subset\mathbb{C}$.
Let $L$ be the oriented path parametrized by a continuous
function $z:\mathbb{R}\to\mathcal{D}$, with orientation
given by increasing $t$.
Suppose that $z$ is $C^1$ except at finitely many points and
\begin{equation}\label{rectifiability}
\int_{-T}^{T}|z'(t)|\,\mathrm{d}t< +\infty
\qquad\text{for every } T>0.
\end{equation}

Suppose that there exists $t_0\in\mathbb{R}$ such that,
with $z_0=z(t_0)$, the following conditions hold:
\begin{itemize}
\item[\textup{(1)}]
$z$ is $C^1$ in a neighborhood of $t_0$\textup{;}

\item[\textup{(2)}]
$f'(z_0)=0$ and $g(z_0)\neq 0$\textup{;}

\item[\textup{(3)}]
$\operatorname{Re}\bigl(f''(z_0)z'(t_0)^2\bigr)<0$\textup{;}

\item[\textup{(4)}]
$\operatorname{Re}\big(f(z(t))\big)<\operatorname{Re}\big(f(z_0)\big)$
for every $t\neq t_0$\textup{;}

\item[\textup{(5)}]
$g_n(z(t))\to g(z(t))$ uniformly on every bounded
interval of $t$ as $n\to +\infty$\textup{;}

\item[\textup{(6)}]
There exist $R>|t_0|$ and $\varepsilon>0$ such that
\[
\int_{|t|>R} \left|e^{nf(z(t))}g_n(z(t))z'(t)\right| \,\mathrm{d}t = O\left(e^{n(\operatorname{Re}f(z_0)-\varepsilon)}\right) \qquad\text{as } n\to+\infty.
\]
\end{itemize}
Then, as $n\to+\infty$,
\[
\frac{1}{2\pi i}\int_L e^{nf(z)}g_n(z)\,\mathrm{d}z \sim \frac{z'(t_0)g(z_0)e^{nf(z_0)}}{i\sqrt{2\pi n}\,\sqrt{-f''(z_0)z'(t_0)^2}},
\]
where the square root in the denominator is the principal
square root.
\end{theorem}

\begin{proof}
Put $v=z'(t_0)$ and $a=-f''(z_0)v^2$, so that
$\operatorname{Re}(a)>0$ by condition (3).
Since $z$ is $C^1$ near $t_0$ and $f'(z_0)=0$ by conditions (1) and (2), we have
\[
f(z(t))
=
f(z_0)-\frac{a}{2}(t-t_0)^2
+o\bigl((t-t_0)^2\bigr)
\qquad\text{as } t\to t_0.
\]
Consequently, there exist $\delta,c>0$ such that
\[
\operatorname{Re}\bigl(f(z(t))-f(z_0)\bigr)
\leqslant -c(t-t_0)^2
\qquad(|t-t_0|\leqslant\delta).
\]
Using condition \textup{(5)}, the substitution
$t=t_0+u/\sqrt{n}$, and dominated convergence, we obtain
\[
\int_{|t-t_0|<\delta}
e^{nf(z(t))}g_n(z(t))z'(t)\,\mathrm{d}t
\sim
\frac{e^{nf(z_0)}}{\sqrt{n}}g(z_0)v
\int_{\mathbb{R}}e^{-au^2/2}\,\mathrm{d}u
=
\frac{\sqrt{2\pi}\,g(z_0)v}
{\sqrt{na}}e^{nf(z_0)}.
\]
Choose $\delta$ small enough that
$[t_0-\delta,t_0+\delta]\subset(-R,R)$.
By condition \textup{(4)} and compactness, there exists
$\varepsilon'>0$ such that
\[
\operatorname{Re}f(z(t))
\leqslant \operatorname{Re}f(z_0)-\varepsilon'
\]
whenever $|t|\leqslant R$ and $|t-t_0|\geqslant\delta$.
Condition \textup{(5)} and the finite length of this
part of the path show that its contribution is
$O(e^{n(\operatorname{Re}f(z_0)-\varepsilon')})$.
Condition \textup{(6)} bounds the remaining tails.
Dividing by $2\pi i$ completes the proof.
\end{proof}

\section{Rational functions and linear forms}
\label{Sect_RL}
	
For any $m \in \mathbb{N}$, the Pochhammer symbol $(t)_m$ is defined by
\[ 
(t)_m := t(t+1)\cdots(t+m-1). 
\]
We adopt the convention $(t)_0 := 1$.
We also define
\[ 
\delta_m := 
\begin{cases}
0 &\quad\text{if $m$ is even}, \\
1 &\quad\text{if $m$ is odd}.
\end{cases} 
\]
	
\begin{definition}[rational functions]
\label{def_R_n(t)}
Let $k \geqslant 2$ and $q \geqslant 3$ be integers. 
Let $r > 2k$ be an integer. 
For any integer $n \in q!\mathbb{N}$, we define the rational function
\[
R_n(t) :=  \frac{(rqn)!}{(qn)!^{2k}} \cdot q^{2kqn} \cdot \left( \prod_{p \mid q \atop p \text{~prime}} p^{2kqn/(p-1)} \right) \cdot (2qt+rqn)^{1-\delta_k} \cdot \frac{(t-qn)_{qn}^k (t+rn+1)_{qn}^{k}}{(qt)_{rqn+1}}.   
\]
\end{definition}
The condition $n \in q!\mathbb{N}$ implies that $R_n(t) \in \mathbb{Q}(t)$ and that $n$ is even. 
	
\bigskip
	
For a rational function of the form $R(t)=P(t)/Q(t)$, where $P(t)$ and $Q(t)$ are polynomials in $t$, we define its degree by $\deg R := \deg P - \deg Q$. 
Then, 
\begin{equation}
\label{deg_R_n<0}
\deg R_n(t) = -\delta_k -(r-2k)qn \leqslant -qn
\end{equation}
since $r > 2k$. 
Therefore, the partial-fraction decomposition of $R_n(t)$ has the form
\begin{equation}
\label{partial_fraction_decomposition}
R_n(t) = \sum_{j=0}^{rqn} \frac{C_{n,j}}{qt+j},
\end{equation}
where the coefficients $C_{n,j}$ ($j \in \{ 0,1,\ldots,rqn \}$) are given by
\begin{align}
C_{n,j} &= R_n(t)(qt+j) \Big|_{t=-\frac{j}{q}} \notag\\
&= (-1)^j \binom{rqn}{j} (rqn-2j)^{1-\delta_k} \notag\\
&\qquad\times \frac{q^{2kqn} \prod_{p \mid q \atop p \text{~prime}} p^{2kqn/(p-1)} \cdot (-j/q-qn)_{qn}^k (-j/q+rn+1)_{qn}^{k}}{(qn)!^{2k}}.  \label{def_C_j}
\end{align}
	
\bigskip
	
\begin{definition}[linear forms]
\label{def_S_n}
Let $k \geqslant 2$ be an integer. 
For any $n \in q!\mathbb{N}$, we define the quantity
\begin{equation}
\label{sum_of_S_n}
S_n := \frac{1}{(k-1)!}\sum_{m=1}^{+\infty} R_n^{(k-1)}(m),
\end{equation}
where $R_n^{(k-1)}(t)$ denotes the $(k-1)$-th derivative of the rational function $R_n(t)$.
\end{definition}
	
\begin{lemma}
\label{lem_linear_forms}
For any $n \in q!\mathbb{N}$, we have
\[  
S_n = \rho_{n,0} +  \rho_{n,1}\delta_k\zeta(k) + \sum_{1 \leqslant a < q/2} \rho_{n,a/q}  \zeta^{-}\left( k, \frac{a}{q} \right),  
\]
where 
\begin{align}
\rho_{n,0} &= \frac{(-1)^k}{q}\left(\sum_{j=1}^{rn}\sum_{m=1}^j\frac{C_{n,qj}}{m^k}+\sum_{a=1}^{q-1}\sum_{j=0}^{rn-1}\sum_{m=0}^j\frac{C_{n,qj+a}}{\left(m+\frac{a}{q}\right)^k}\right),  \label{def_rho_n_0}\\
\rho_{n,a/q}  &=  \frac{(-1)^{k-1}}{q}\sum_{j=0}^{rn-1}C_{n,qj+a} \quad (1 \leqslant a < q), \label{def_rho_n_a/q}\\
\rho_{n,1} &= \frac{(-1)^{k-1}}{q}\sum_{j=0}^{rn}C_{n,qj}+(1-\delta_q)(2^k-1)\rho_{n,1/2}. \label{def_rho_n_1}
\end{align}
\textup{(}On the right-hand side of \eqref{def_rho_n_1}, the term $\rho_{n,1/2}$ is defined by \eqref{def_rho_n_a/q} with $a=q/2$ when $q$ is even.\textup{)}
\end{lemma}
	
\begin{proof}
Applying the differential operator $(1/(k-1)!)\mathrm{d}^{k-1}/\mathrm{d} t^{k-1}$ to \eqref{partial_fraction_decomposition}, we obtain
\begin{equation}
\label{(k-1)_derivative_of_R_n(t)}
\frac{1}{(k-1)!} R_n^{(k-1)}(t)= (-1)^{k-1} q^{k-1} \sum_{j=0}^{rqn} \frac{C_{n,j}}{(qt+j)^k}=\frac{(-1)^{k-1}}{q}\sum_{j=0}^{rqn} \frac{C_{n,j}}{\left(t+\frac{j}{q}\right)^k}.
\end{equation}
Specializing \eqref{(k-1)_derivative_of_R_n(t)} at $t=m \in \mathbb{N}$ and taking the sum over all $m \in \mathbb{N}$, we have
\begin{align*}
S_n &= \frac{1}{(k-1)!}\sum_{m=1}^{+\infty} R_n^{(k-1)}(m) = \frac{(-1)^{k-1}}{q}\left(\sum_{j=0}^{rqn}C_{n,j}\sum_{m=1}^{+\infty}\frac{1}{\left(m+\frac{j}{q}\right)^k}\right)\\
&=\frac{(-1)^{k-1}}{q}\left(\sum_{j=0}^{rn}C_{n,qj}\sum_{m=1}^{+\infty}\frac{1}{\left(m+j\right)^k}+\sum_{a=1}^{q-1}\sum_{j=0}^{rn-1}C_{n,qj+a}\sum_{m=1}^{+\infty}\frac{1}{\left(m+j+\frac{a}{q}\right)^k}\right)\\
&=\frac{(-1)^{k-1}}{q}\left(C_{n,0}\zeta(k)+\sum_{j=1}^{rn}C_{n,qj}\left(\zeta(k)-\sum_{m=1}^j\frac{1}{m^k}\right)\right)\\
&\quad+\frac{(-1)^{k-1}}{q}\sum_{a=1}^{q-1}\sum_{j=0}^{rn-1}C_{n,qj+a}\left(\zeta\left(k,\frac{a}{q}\right)-\sum_{m=0}^{j}\frac{1}{\left(m+\frac{a}{q}\right)^k}\right)\\
&=\rho_{n,0}+\rho_{n,1}'\zeta(k)+\sum_{a=1}^{q-1}\rho_{n,a/q}\zeta\left(k,\frac{a}{q}\right),
\end{align*}
where
\begin{align*}
\rho_{n,0}&=\frac{(-1)^k}{q}\left(\sum_{j=1}^{rn}\sum_{m=1}^j\frac{C_{n,qj}}{m^k}+\sum_{a=1}^{q-1}\sum_{j=0}^{rn-1}\sum_{m=0}^j\frac{C_{n,qj+a}}{\left(m+\frac{a}{q}\right)^k}\right),\\
\rho_{n,1}'&=\frac{(-1)^{k-1}}{q}\sum_{j=0}^{rn}C_{n,qj},\\
\rho_{n,a/q}&=\frac{(-1)^{k-1}}{q}\sum_{j=0}^{rn-1}C_{n,qj+a}.
\end{align*}
By \eqref{def_C_j}, we have the following symmetry property:
\begin{equation}
\label{symmetry}
C_{n,j} = (-1)^{k-1}  C_{n,rqn-j}, \quad j \in \{0,1,\ldots,rqn\}.
\end{equation} 
(We have used that $n$ is even and $\delta_k \equiv k \pmod{2}$.) 
Therefore, for any integer $a$ with $1 \leqslant a < q$, we have
\begin{align*}
\rho_{n,1}'&=\frac{(-1)^{k-1}}{q}\sum_{j=0}^{rn}C_{n,qj}=\frac{(-1)^{k-1}}{q}\sum_{j=0}^{rn}C_{n,rnq-qj}\\
&=(-1)^{k-1}\cdot\frac{(-1)^{k-1}}{q}\sum_{j=0}^{rn}C_{n,qj}=(-1)^{k-1}\rho_{n,1}',
\end{align*}
and
\begin{align*}
\rho_{n,1-a/q}&=\frac{(-1)^{k-1}}{q}\sum_{j=0}^{rn-1}C_{n,qj+q-a}=\frac{(-1)^{k-1}}{q}\sum_{j=0}^{rn-1}C_{n,rnq-jq-a}\\
&=(-1)^{k-1}\cdot\frac{(-1)^{k-1}}{q}\sum_{j=0}^{rn-1}C_{n,jq+a}=(-1)^{k-1}\cdot\rho_{n,a/q}.
\end{align*}
In particular, if $q$ is even, then we have $\rho_{n,1/2}=(-1)^{k-1}\rho_{n,1/2}$. 
We conclude that
\begin{align*}
S_n &= \rho_{n,0} +\rho_{n,1}'\zeta(k)+ \sum_{a=1}^{q-1} \rho_{n,a/q} \zeta\left( k,\frac{a}{q} \right) \\
&=\rho_{n,0} +\rho_{n,1}'\delta_k\zeta(k)+\sum_{1\leqslant a<q/2}\rho_{n,a/q}\zeta^-\left( k,\frac{a}{q} \right)+ 
\begin{cases}
0 \quad&\text{if $q$ is odd}, \\
\rho_{n,1/2}\delta_k\zeta\left(k,\frac{1}{2}\right) \quad&\text{if $q$ is even},
\end{cases}\\
&=\rho_{n,0} +\rho_{n,1}'\delta_k\zeta(k)+\sum_{1\leqslant a<q/2}\rho_{n,a/q}\zeta^-\left( k,\frac{a}{q} \right)+ \begin{cases}
0 \quad&\text{if $q$ is odd}, \\
\rho_{n,1/2}(2^k-1)\delta_k\zeta\left(k\right) \quad&\text{if $q$ is even},
\end{cases}\\
&=\rho_{n,0} +\rho_{n,1}\delta_k\zeta(k)+\sum_{1\leqslant a<q/2}\rho_{n,a/q}\zeta^-\left( k,\frac{a}{q} \right),
\end{align*}
where
\[
\rho_{n,1}=\rho_{n,1}'+(1-\delta_q)(2^k-1)\rho_{n,1/2}.
\]
The proof of Lemma \ref{lem_linear_forms} is complete.
\end{proof}
	
\bigskip
	
As usual, we denote by 
\[
d_m := \operatorname{lcm}\{1,2,\ldots,m\} 
\]
the least common multiple of $1,2,\ldots,m$ for any positive integer $m$.  
	
\begin{lemma}
\label{lem_beta}
For any $n \in q!\mathbb{N}$, we have 
\[ 
q \cdot \rho_{n,1} \in \mathbb{Z},  \quad q \cdot \rho_{n,a/q} \in \mathbb{Z} ~~(1 \leqslant a < q/2),  \quad d_{rqn}^{k} \cdot \rho_{n,0} \in \mathbb{Z}. 
\]
Moreover, we have
\[ 
\max \left\{ |\rho_{n,0}|, |\rho_{n,1}|, |\rho_{n,a/q}| ~\Big|~ 1 \leqslant a < q/2 \right\} \leqslant \exp\left( \beta n + o(n) \right) \quad \text{as~} n \to +\infty,  
\]
where
\begin{equation}
\label{def_beta}
\beta = rq\log 2 + k \left(  (2q+r)\log\left(q+\frac{r}{2}\right) - r\log \frac{r}{2} + 2q\sum_{p \mid q \atop p \text{~prime}} \frac{\log p}{p-1}  \right). 
\end{equation}
\end{lemma}
	
\begin{proof}
First, by \eqref{def_C_j}, the coefficients $C_{n,j}$ ($j=0,1,\ldots,rqn$) can be expressed as
\begin{equation}
\label{C_j_another_expression}
C_{n,j} = (-1)^j \binom{rqn}{j} (rqn-2j)^{1-\delta_k} \cdot A_{n,j}^k \cdot B_{n,j}^k, 
\end{equation}
where
\begin{align*}
A_{n,j} &= \prod_{p \mid q \atop p \text{~prime}} p^{qn/(p-1)} \cdot \frac{\prod_{\nu=0}^{qn-1} (-j-q^{2}n +q\nu) }{(qn)!}, \\
B_{n,j} &= \prod_{p \mid q \atop p \text{~prime}} p^{qn/(p-1)} \cdot \frac{\prod_{\nu=0}^{qn-1} (-j+rqn +q^{2}n - q\nu) }{(qn)!}.
\end{align*}
By considering the $\ell$-adic valuation of $A_{n,j}$ and $B_{n,j}$ for every prime $\ell$, we obtain the elementary conclusion (see, for example, \cite[Lemma 4.2]{Lai2026}) that
\[ 
A_{n,j} \in \mathbb{Z}, \quad B_{n,j} \in \mathbb{Z}. 
\]
Therefore, we have
\begin{equation}
\label{C_j_is_integral}
C_{n,j} \in \mathbb{Z}, \quad j \in \{0,1,\ldots,rqn\}. 
\end{equation}
By \eqref{def_rho_n_0}, \eqref{def_rho_n_a/q}, \eqref{def_rho_n_1}, and \eqref{C_j_is_integral}, we obtain immediately that
\[ 
q \cdot \rho_{n,1} \in \mathbb{Z},  \quad q \cdot \rho_{n,a/q} \in \mathbb{Z} ~~(1 \leqslant a < q/2),  \quad d_{rqn}^{k} \cdot \rho_{n,0} \in \mathbb{Z}. 
\]
		
Now, noting that
\begin{align*}
\binom{rqn}{j} &\leqslant 2^{rqn}, \qquad |rqn-2j|^{1-\delta_k} \leqslant rqn, \\
|A_{n,j}B_{n,j}| &= \frac{\prod_{p \mid q \atop p \text{~prime}} p^{ 2qn/(p-1)}}{(qn)!^2} \cdot \prod_{\nu=0}^{qn-1} (j+q^{2}n-q\nu)(-j+rqn+q^{2}n-q\nu) \\
&\leqslant \frac{\prod_{p \mid q \atop p \text{~prime}} p^{ 2qn/(p-1)}}{(qn)!^2} \cdot \prod_{\nu=0}^{qn-1} \left(  q^2n + \frac{rqn}{2} - q\nu \right)^2 \\
&= \frac{\prod_{p \mid q \atop p \text{~prime}} p^{ 2qn/(p-1)}}{(qn)!^2} \cdot q^{2qn} \cdot \frac{\Gamma((q+r/2)n+1)^2}{\Gamma((r/2)n+1)^2},
\end{align*}
we deduce from \eqref{C_j_another_expression} that
\begin{equation*}
\max_{0 \leqslant j \leqslant rqn} |C_{n,j}| \leqslant rqn \cdot 2^{rqn} \cdot \left( q^{2qn} \cdot \prod_{p \mid q \atop p \text{~prime}} p^{ 2qn/(p-1)} \cdot \frac{\Gamma((q+r/2)n+1)^2}{\Gamma(qn+1)^2\Gamma((r/2)n+1)^2} \right)^k.
\end{equation*}
Applying Stirling's formula to Gamma values, we obtain
\begin{equation}
\label{extimate_for_C_j}
\max_{0 \leqslant j \leqslant rqn} |C_{n,j}| \leqslant \exp\left( \beta n + o(n) \right) \quad \text{as~} n \to +\infty,   
\end{equation}
where the constant $\beta$ is given by \eqref{def_beta}:
\[ 
\beta = rq\log 2 + k \left(  (2q+r)\log\left(q+\frac{r}{2}\right) - r\log \frac{r}{2} + 2q\sum_{p \mid q \atop p \text{~prime}} \frac{\log p}{p-1}  \right).  
\]
Finally, Equations \eqref{def_rho_n_0}, \eqref{def_rho_n_a/q}, and \eqref{def_rho_n_1} imply that
\[ 
\max \left\{ |\rho_{n,0}|, |\rho_{n,1}|, |\rho_{n,a/q}| ~\Big|~ 1 \leqslant a < q/2  \right\} \leqslant q^{k-1}(rqn+1)^2 \cdot  \max_{0 \leqslant j \leqslant rqn} |C_{n,j}|.  
\]
Therefore, the estimate \eqref{extimate_for_C_j} implies that
\[ 
\max \left\{ |\rho_{n,0}|, |\rho_{n,1}|, |\rho_{n,a/q}| ~\Big|~ 1 \leqslant a < q/2  \right\} \leqslant \exp\left( \beta n + o(n) \right) \quad \text{as~} n \to +\infty. 
\]
The proof of Lemma \ref{lem_beta} is complete.
\end{proof}
	
\section{A property of the Hurwitz zeta values}
\label{Sect_Prop_of_Hurw}
	
The goal of this section is to prove that the linear forms $S_n$ (see Lemma \ref{lem_linear_forms}) belong to the space $\mathbb{Q} + V_k^{-}(q)$. 
By Definition \ref{def_V_k(q)} and the simple fact
\[ 
\zeta^{-}\left( k, 1-\frac{a}{q} \right) = (-1)^{k-1}  \zeta^{-}\left( k, \frac{a}{q} \right), \quad \zeta^{+}\left( k, 1-\frac{a}{q} \right) = (-1)^{k}  \zeta^{+}\left( k, \frac{a}{q} \right), 
\]
it is easy to see that 
\begin{align*}
V_k^{-}(q) &= \operatorname{Span}_{\mathbb{Q}}\left\{  \zeta^{-}\left(k,\frac{a}{q}\right)  ~~\Big|~~ 1 \leqslant a < q,~\gcd(a,q)=1 \right\} ~(k \geqslant 2,~ q \geqslant 3),\\
V_k^{+}(q) &= \operatorname{Span}_{\mathbb{Q}}\left\{  \zeta^{+}\left(k,\frac{a}{q}\right)  ~~\Big|~~ 1 \leqslant a < q,~\gcd(a,q)=1 \right\} ~(k \geqslant 2,~ q \geqslant 3).
\end{align*}
For convenience, we define these spaces for \(k\geqslant2\) and \(q=2\). By \eqref{eq:half-argument}, we have
\begin{align*}
V_k^{-}(2) &:= \operatorname{Span}_{\mathbb{Q}}\left\{  \zeta^{-}\left(k,\frac{1}{2}\right) \right\} = \delta_k\zeta(k) \mathbb{Q}, \\
V_k^{+}(2) &:= \operatorname{Span}_{\mathbb{Q}}\left\{  \zeta^{+}\left(k,\frac{1}{2}\right) \right\} = (1-\delta_k)\zeta(k) \mathbb{Q}, \\
V_k(2) &:= \operatorname{Span}_{\mathbb{Q}}\left\{  \zeta\left(k,\frac{1}{2}\right) \right\} = \zeta(k) \mathbb{Q}.
\end{align*}
	
\begin{lemma}
\label{lem_a_property_of_Hurwitz_zeta_values}
Let $k\geqslant2$ and $q'\geqslant 2$ be integers. 
Let $q=pq'$, where $p$ is a prime number. 
Then, we have
\[
V_k^{-}(q')\subset V_k^{-}(q) \quad\text{and}\quad V_k^{+}(q')\subset V_k^{+}(q).
\]
\end{lemma}
	
\begin{proof}
It suffices to prove that $\zeta^{-}(k,a/q') \in V_k^{-}(q)$ and $\zeta^{+}(k,a/q') \in V_k^{+}(q)$ for any integer $a$ such that $1\leqslant a<q'$ and $\gcd(a,q')=1$. 
We fix such an integer $a$.
		
Since
\[
a+q'\mathbb{Z}_{\geqslant 0}=\bigsqcup_{j=0}^{p-1}(a+jq'+q\mathbb{Z}_{\geqslant 0})\ ,\quad q'-a+q'\mathbb{Z}_{\geqslant 0}=\bigsqcup_{j=0}^{p-1}(q'-a+jq'+q\mathbb{Z}_{\geqslant 0}),
\]
we have the following distribution formulae for Hurwitz zeta values: 
\begin{align*}
p^k\zeta\left(k,\frac{a}{q'}\right)&=\sum_{j=0}^{p-1}\zeta\left(k,\frac{a+jq'}{q}\right), \\
p^k\zeta\left(k,1-\frac{a}{q'}\right)&=\sum_{j=0}^{p-1}\zeta\left(k,1-\frac{a+jq'}{q}\right). 
\end{align*}
Therefore, we have
\begin{align}
p^k\zeta^{-}\left(k,\frac{a}{q'}\right)&=\sum_{j=0}^{p-1}\zeta^{-}\left(k,\frac{a+jq'}{q}\right), \label{distribution_formula_odd}\\
p^k\zeta^{+}\left(k,\frac{a}{q'}\right)&=\sum_{j=0}^{p-1}\zeta^{+}\left(k,\frac{a+jq'}{q}\right). \label{distribution_formula_even}
\end{align}
Now, we distinguish between two cases.
		
\textbf{Case 1}: $p \mid q'$. 
In this case, we have
\[ 
\gcd(a+jq',q) = 1 \quad \text{for all~} j \in \{0,1,\ldots,p-1\},
\]
because $\gcd(a+jq',p) \mid \gcd(a+jq',q')=1$. 
Hence, each summand on the right-hand side of \eqref{distribution_formula_odd} (resp., \eqref{distribution_formula_even}) belongs to $V_k^{-}(q)$ (resp., $V_k^{+}(q)$). 
We obtain
\[ 
\zeta^{-}\left(k,\frac{a}{q'}\right) \in V_k^{-}(q) \quad\text{and}\quad  \zeta^{+}\left(k,\frac{a}{q'}\right) \in V_k^{+}(q). 
\]
		
\textbf{Case 2}: $p \nmid q'$. 
In this case, there exists a unique integer $j_0 \in \{0,1,\ldots,p-1\}$ such that $p \mid (a+j_0 q')$. 
For any $j \in \{ 0,1,\ldots,p-1 \} \setminus \{ j_0 \}$, we have $\gcd(a+j q',q)=1$. 
Hence, we deduce from \eqref{distribution_formula_odd} that
\[ 
p^k\zeta^{-}\left(k,\frac{a}{q'}\right) - \zeta^{-}\left(k,\frac{a+j_0q'}{q}\right) =  \sum_{0\leqslant j\leqslant p-1 \atop j\neq j_0}\zeta^{-}\left(k,\frac{a+jq'}{q}\right) \in V_{k}^{-}(q). 
\]
In other words, we have
\[ 
p^k\zeta^{-}\left(k,\frac{a}{q'}\right) - \zeta^{-}\left(k,\frac{(a+j_0q')/p}{q'}\right) \in V_{k}^{-}(q). 
\]
Write $a_1 = (a+j_0q')/p$. 
Then $a_1$ is an integer such that $1 \leqslant a_1 < q'$ and $\gcd(a_1,q')=1$. We have
\[ 
\zeta^{-}\left(k,\frac{a_1}{q'}\right) \equiv p^{k} \zeta^{-}\left(k,\frac{a}{q'}\right) \pmod{V_{k}^{-}(q)}. 
\]
Moreover, we have
\[ 
pa_1 \equiv a \pmod{q'}. 
\]
Now, there exists a unique integer $j_1 \in \{ 0,1,\ldots,p-1 \}$ such that $p \mid (a_1 + j_1q')$. 
Repeating the arguments above, we find that $a_2 = (a_1+j_1 q')/p$ is an integer such that $1 \leqslant a_2 < q'$, $\gcd(a_2,q')=1$, 
\[ 
\zeta^{-}\left(k,\frac{a_2}{q'}\right) \equiv p^{k} \zeta^{-}\left(k,\frac{a_1}{q'}\right) \equiv p^{2k} \zeta^{-}\left(k,\frac{a}{q'}\right) \pmod{V_{k}^{-}(q)}, 
\]
and 
\[ 
p^2a_2 \equiv p a_1 \equiv a \pmod{q'}. 
\]
Continuing in this way, we obtain a sequence of integers $\{ a_n \}_{n \geqslant 1}$ such that $1 \leqslant a_n < q'$, $\gcd(a_n,q')=1$,
\[ 
\zeta^{-}\left(k,\frac{a_n}{q'}\right) \equiv p^{nk} \zeta^{-}\left(k,\frac{a}{q'}\right) \pmod{V_{k}^{-}(q)}, 
\]
and
\[ 
p^n a_n \equiv a \pmod{q'}, 
\]
for any $n \geqslant 1$. 
In particular, we have $a_{\varphi(q')} = a$ and
\[ 
\zeta^{-}\left(k,\frac{a}{q'}\right) \equiv p^{\varphi(q')k} \zeta^{-}\left(k,\frac{a}{q'}\right) \pmod{V_{k}^{-}(q)}, 
\]
which implies that 
\[ 
\zeta^{-}\left(k,\frac{a}{q'}\right) \in V_{k}^{-}(q). 
\]
Similarly, we have $ \zeta^{+}\left(k,a/q'\right) \in V_{k}^{+}(q)$ for Case 2. 
The proof of Lemma \ref{lem_a_property_of_Hurwitz_zeta_values} is complete.
\end{proof}
		
\begin{corollary}
\label{lem_S_n_is_in_Q+V^odd}
Let $k\geqslant2$ and $q\geqslant3$ be integers. 
\begin{enumerate}
\item[\textup{(1)}] For any divisor $q' \geqslant 2$ of $q$, we have
\[
V^{-}_k(q')\subset V^{-}_k(q), \quad  V^{+}_k(q')\subset V^{+}_k(q), \quad V_k(q') \subset V_k(q). 
\]
\item[\textup{(2)}] For any integer $a \in \{1,2,\ldots,q-1\}$ \textup{(}not necessarily coprime to $q$\textup{)}, we have
\[
\zeta^{-}\left(k,\frac{a}{q}\right)\in V_k^{-}(q), \quad \zeta^{+}\left(k,\frac{a}{q}\right)\in V_k^{+}(q), \quad \zeta\left(k,\frac{a}{q}\right)\in V_k(q).
\]
\item[\textup{(3)}] We have
\[ 
\operatorname{Span}_{\mathbb{Q}}\left( \left\{ 1,\delta_k\zeta(k) \right\} \bigcup \left\{  \zeta^-\left( k, \frac{a}{q} \right) ~\Big|~ a \in \mathbb{Z},\ 1 \leqslant a < \frac{q}{2} \right\} \right) = \mathbb{Q} + V_k^-(q). 
\]
\end{enumerate}
\end{corollary}
		
\begin{proof}
Applying Lemma \ref{lem_a_property_of_Hurwitz_zeta_values} repeatedly, we obtain
\[ 
V^{-}_k(q')\subset V^{-}_k(q) \quad\text{and}\quad V^{+}_k(q')\subset V^{+}_k(q) 
\]
for any divisor $q' \geqslant 2$ of $q$. 
Since $V_k(q)=V_k^{-}(q) + V_k^{+}(q)$, we also have $V_k(q') \subset V_k(q)$. 
The first assertion (1) is proved.
			
For any integer $a \in \{ 1,2,\ldots,q-1 \}$, let $a'=a/\gcd(a,q)$ and $q'=q/\gcd(a,q)$. 
Then $a/q = a'/q'$ and $\gcd(a',q')=1$. 
Clearly $q' \geqslant 2$. 
Thus, $\zeta^{-}(k,a/q) \in V_k^{-}(q')$, $\zeta^{+}(k,a/q) \in V_k^{+}(q')$, and $\zeta(k,a/q) \in V_k(q')$. 
Therefore, assertion (2) follows from assertion (1).
			
For the last assertion (3), it remains to prove that $\delta_k \zeta(k) \in V_k^{-}(q)$. 
If $k$ is even, then $\delta_k=0$ and there is nothing to prove. 
If $k$ is odd, then by assertion (2) we have
\[
\zeta(k) = \frac{1}{q^k-1} \sum_{a=1}^{q-1} \zeta\left(k, \frac{a}{q} \right) = \frac{1}{2(q^k-1)} \sum_{a=1}^{q-1} \zeta^{-}\left(k,\frac{a}{q}\right) \in V_k^{-}(q),
\]
which completes the proof of Corollary \ref{lem_S_n_is_in_Q+V^odd}.
\end{proof}
		
\section{Integral representations of $S_n$}
\label{Sect_Int_expre_of_S_n}
		
In this section, we will present the linear forms $S_n$ (see Lemma \ref{lem_linear_forms}) as complex integrals. 
This serves as a preparatory step for applying the saddle-point method to estimate $S_n$. Throughout this section, we assume that $k,q,r$ are positive integers with $k \geqslant 2$, $q \geqslant 3$, and $r>2k$.
		
Following Zudilin \cite{Zud2002}, we define `differential iterations' of the cotangent function $\cot(z)$.
		
\begin{definition}
For any integer $k \geqslant 2$, we define 
\[ 
\cot_k(z) := \frac{(-1)^{k-1}}{(k-1)!}\cdot \frac{\mathrm{d}^{k-1}\cot(z)}{\mathrm{d}z^{k-1}}. 
\]
\end{definition}			
		
\bigskip
		
The following lemma summarizes the basic properties of the function $\cot_k(z)$.
		
\begin{lemma}
\label{property_of_cot_k}
Let $k \geqslant 2$ be an integer.
\begin{enumerate}
\item[\textup{(1)}] The function $\cot_{k} (\pi z)$ is meromorphic on $\mathbb{C}$. 
The set of poles of $\cot_{k} (\pi z)$ is exactly $\mathbb{Z}$. 
For any $m \in \mathbb{Z}$, we have
\begin{equation*}
\pi^k\cot_k(\pi z)=\frac{1}{(z-m)^k}+O(1)
\end{equation*}
in a small neighborhood of $z=m$.
				
\item[\textup{(2)}] For any $z \in \mathbb{C} \setminus \mathbb{Z}$, we have
\[ 
|\pi^k \cot_k(\pi z)| \leqslant \frac{2}{\operatorname{dist}(z,\mathbb{Z})^k} + 4, 
\]
where $\operatorname{dist}(z,\mathbb{Z}) = \inf\limits_{m \in \mathbb{Z}} |z-m|$ is the distance between $z$ and $\mathbb{Z}$.
				
\item[\textup{(3)}] There exist rational constants $c_l$ \textup{(}$0 \leqslant l \leqslant k-2$, $l \equiv k \pmod{2}$\textup{)} depending only on $k$ such that
\[ 
c_{k-2} \neq 0 
\]
and
\[ 
\sin^k(\pi z)\cdot\cot_k(\pi z) = \sum_{0 \leqslant l \leqslant k-2 \atop l \equiv k \pmod{2}} c_l \cos(l \pi z), \quad z \in \mathbb{C} \setminus \mathbb{Z}.  
\]
\end{enumerate} 
\end{lemma}
		
\begin{proof}
For part (1), see \cite[Lemma 2.3]{Zud2002}. 
Next, we prove part (2). 
It is well known that
\[ 
\pi \cot(\pi z) = \frac{1}{z} + \sum_{m \in \mathbb{Z}\setminus\{0\}} \left( \frac{1}{z-m} + \frac{1}{m} \right), 
\]
where the series on the right-hand side converges absolutely and uniformly on every compact subset of $\mathbb{C} \setminus \mathbb{Z}$ (see, for instance, \cite[Example 2.4, p. 379]{Lan1999}). 
Therefore, we have
\[ 
\pi^k \cot_k(\pi z) = \sum_{m \in \mathbb{Z}} \frac{1}{(z-m)^k}, \quad z \in \mathbb{C}\setminus\mathbb{Z}. 
\]
(The above equation gives an alternative proof of part (1).)
Let $x = \operatorname{Re} z$. 
For any integer $m > \lceil x \rceil $, we have $|z-m| \geqslant |\operatorname{Re}(z-m)| \geqslant m - \lceil x \rceil$. 
For any integer $m < \lfloor x \rfloor$, we have $|z-m| \geqslant |\operatorname{Re}(z-m)| \geqslant  \lfloor x \rfloor - m$. 
Hence,
\[
|\pi^k \cot_k (\pi z)| \leqslant 2 \cdot \frac{1}{\operatorname{dist}(z,\mathbb{Z})^k} + 2\cdot \zeta(k) < \frac{2}{\operatorname{dist}(z,\mathbb{Z})^k} + 4, 
\]
which proves part (2). 
Finally, we prove part (3).
By \cite[Lemma 2.2]{Zud2002}, there exists a polynomial $W_k(X) \in \mathbb{Q}[X]$ depending only on $k$ such that
\[ 
\sin^k(z)\cdot\cot_k(z)=W_k(\cos(z)),\quad W_k(-X)=(-1)^kW_k(X),\quad \deg W_k=k-2. 
\]
In other words, there exist rational constants $\widetilde{c}_l$ ($0 \leqslant l \leqslant k-2$, $l \equiv k \pmod{2}$) depending only on $k$ such that
\[ 
\widetilde{c}_{k-2} \neq 0 
\]
and
\[ 
\sin^k(\pi z)\cdot\cot_k(\pi z) = \sum_{0 \leqslant l \leqslant k-2 \atop l \equiv k \pmod{2}} \widetilde{c}_l \cos^{l} (\pi z). 
\]
By expanding
\[ 
\cos^{l} (\pi z) = \left( \frac{e^{i\pi z} + e^{-i\pi z}}{2} \right)^l, 
\]
we see that part (3) holds.
\end{proof}
		
\bigskip
		
Next, we express the linear forms $S_n$ (see Definition \ref{def_S_n} and Lemma \ref{lem_linear_forms}) as complex integrals.
		
\begin{lemma}
\label{S_n=int_cot_R}
For any $n \in q!\mathbb{N}$ and any $M \in (0,qn)$, we have
\[ 
S_n = \frac{\pi^{k-1}i}{2} \int_{M-i\infty}^{M+i\infty}  \cot_k(\pi z) R_n(z) \,\de z. 
\]
\end{lemma}
		
\begin{proof}
By Definition \ref{def_R_n(t)} and Lemma \ref{property_of_cot_k} (1), the function $\cot_k(\pi z) R_n(z)$ is meromorphic on $\mathbb{C}$.
Fix any $n \in q!\mathbb{N}$ and $M \in (0,qn)$. 
Let $T>qn$ be a sufficiently large real number. 
Consider the anti-clockwise rectangular contour $\mathcal{R}_T$ with vertices at $M \pm iT$ and $\lfloor T \rfloor +1/2 \pm iT$. 
Note that for every $m\in\{1,\ldots,qn\}$, the factor $(z-qn)_{qn}^k$ shows that $R_n(z)$ has a zero of order at least
$k$ at $z=m$. 
Since $\cot_k(\pi z)$ has a pole of order $k$ there, the product
\[
\cot_k(\pi z)R_n(z)
\]
extends holomorphically across these points.
In particular, the vertical line $\operatorname{Re}z=M$
contains no pole of $\cot_k(\pi z)R_n(z)$, even when $M$ is an integer.
By Cauchy's residue formula, we have
\begin{align*}
\frac{1}{2\pi i} \int_{\mathcal{R}_T}  \cot_k(\pi z) R_n(z) \,\de z &= \sum_{m=\lceil M \rceil}^{\lfloor T \rfloor} \operatorname{Res}_{z=m}\left( \cot_k(\pi z) R_n(z) \right) \\
&= \sum_{m=qn+1}^{\lfloor T \rfloor} \operatorname{Res}_{z=m}\left( \cot_k(\pi z) R_n(z) \right),
\end{align*}
where the last equality used $R_n^{(k-1)}(m) = 0$ for $m \in \{1,2,\ldots,qn\}$.
In a small neighborhood of $m \in \mathbb{N}$, Lemma \ref{property_of_cot_k} (1) implies that
\begin{align*}
\cot_k&(\pi z)R_n(z) = \left( \frac{1}{\pi^k(z-m)^k} + O(1) \right) \\
&\times \left( R_n(m) + \frac{R_n^{\prime}(m)}{1!}(z-m) + \cdots + \frac{R_n^{(k-1)}(m)}{(k-1)!}(z-m)^{k-1} + O(|z-m|^k) \right). 
\end{align*}
Therefore, we have $\operatorname{Res}_{z=m}\left( \cot_k(\pi z) R_n(z) \right)= \pi^{-k} R_n^{(k-1)}(m)/(k-1)!$ and
\begin{equation}
\label{541}
\frac{1}{2\pi i} \int_{\mathcal{R}_T}  \cot_k(\pi z) R_n(z) \,\de z = \frac{1}{\pi^{k}(k-1)!}  \sum_{m=qn+1}^{\lfloor T \rfloor} R_n^{(k-1)}(m).
\end{equation}
			
For any complex number $z$ on the three sides $[M-iT, \lfloor T \rfloor +1/2 - iT]$, $[\lfloor T \rfloor +1/2 - iT, \lfloor T \rfloor +1/2 + iT]$, and $[\lfloor T \rfloor +1/2 + iT, M+iT]$ of the rectangle, we have $\operatorname{dist}(z,\mathbb{Z}) \geqslant 1/2$. 
By Lemma \ref{property_of_cot_k} (2) and Equation \eqref{deg_R_n<0}, we have
\[
|\cot_k(\pi z)| \leqslant \frac{2^{k+1}+4}{\pi^k} \quad\text{and}\quad |R_n(z)| = O(T^{-2}),
\]
where the implicit constant depends only on $k,q,r,n$. Therefore, 
\[
\left( \int_{M-iT}^{\lfloor T \rfloor +1/2 - iT} + \int_{\lfloor T \rfloor +1/2 - iT}^{\lfloor T \rfloor +1/2 + iT} + \int_{\lfloor T \rfloor +1/2 + iT}^{M+iT}\right)| \cot_k(\pi z)R_n(z)|\,|\de z| = O(T^{-1}).
\]
Substituting the above estimate into \eqref{541} and letting $T \to +\infty$, we obtain the desired integral representation for $S_n$.
\end{proof}
		
\bigskip
		
\begin{definition}
\label{def_f_g}
We define two holomorphic functions on $\mathbb{C} \setminus \big( (-\infty,0] \cup [q,+\infty) \big)$ as follows.
\begin{align}
f(z)&:=k(z+r+q)\log(z+r+q)+k(q-z)\log(q-z) \notag\\
&\quad+(q+k)z\log z-(q+k)(z+r)\log(z+r) \notag\\
&\quad+rq\log r+2kq\sum_{p \mid q \atop p \text{~prime}}\frac{\log p}{p-1},\label{def_f}\\
g(z)&:=\frac{(2z+r)^{1-\delta_k}}{\sqrt{z}\sqrt{z+r}}\left(\frac{\sqrt{q-z}\sqrt{z+r+q}}{\sqrt{z}\sqrt{z+r}}\right)^{k}\label{def_g}, 
\end{align}
where $\sqrt{(\cdot)} = \exp(\log(\cdot)/2)$ is defined on $\mathbb{C} \setminus (-\infty,0]$. 
Recall that $\log(\cdot)$ denotes the principal branch of the logarithm throughout this paper. 
\end{definition}
		
\bigskip
		
Recall that the Log Gamma function $\log \Gamma(\cdot)$ is a holomorphic function on $\mathbb{C} \setminus (-\infty,0]$ defined by
\[
\log \Gamma(z) := -\gamma z -\log z + \sum_{m=1}^{+\infty} \left( \frac{z}{m} - \log\left(1+\frac{z}{m}\right)\right),
\]
where $\gamma = 0.577\ldots$ is the Euler--Mascheroni constant.
		
\begin{lemma}[A version of Stirling's formula]
\label{lem_Stirling}
For any $z \in \mathbb{C} \setminus (-\infty,0]$, we have
\[
\left| \log \Gamma(z) - \left( \left( z- \frac{1}{2} \right)\log z - z + \frac{\log(2\pi)}{2} \right)\right| \leqslant \frac{\pi}{8} \cdot \frac{1}{\operatorname{dist}(z,\mathbb{R}_{\leqslant 0})},
\]
where $\operatorname{dist}(z,\mathbb{R}_{\leqslant 0})$ denotes the distance between $z$ and $(-\infty,0]$.
\end{lemma}
		
\begin{proof}
By \cite[Equation ($\Gamma 13$), p.\,423]{Lan1999} and \cite[Lemma 2.2, p.\,425]{Lan1999}, we have
\[
\log \Gamma(z) = \left( z- \frac{1}{2} \right)\log z - z + \frac{\log(2\pi)}{2} - \frac{1}{2}\int_{0}^{+\infty} \frac{\{t\}^2-\{t\}}{(z+t)^2} \,\de t \quad\text{for any}~ z \in \mathbb{C} \setminus(-\infty,0],
\]
where $\{t\}$ denotes the fractional part of a real number $t$.
Since $|\{t\}^2-\{t\}| \leqslant 1/4$, it is sufficient to prove that
\[
\int_{0}^{+\infty} \frac{\de t}{|z+t|^2} \leqslant \frac{\pi}{\operatorname{dist}(z,\mathbb{R}_{\leqslant 0})} \quad\text{for any}~z \in \mathbb{C} \setminus(-\infty,0].
\]
Write $z=x+iy$, where $x,y \in \mathbb{R}$.
If $x \leqslant 0$, then $\operatorname{dist}(z,\mathbb{R}_{\leqslant 0}) = |y|$, and
\[
\int_{0}^{+\infty} \frac{\de t}{|z+t|^2} =\left(\int_{0}^{-x}+\int_{-x}^{+\infty}\right) \frac{\de t}{(t+x)^2+y^2} \leqslant 2\int_{0}^{+\infty} \frac{\de t}{t^2+y^2} = \frac{\pi}{|y|} = \frac{\pi}{\operatorname{dist}(z,\mathbb{R}_{\leqslant 0})}.
\]
If $x>0$, then $\operatorname{dist}(z,\mathbb{R}_{\leqslant 0}) = \sqrt{x^2+y^2}$, and
\[
\int_{0}^{+\infty} \frac{\de t}{|z+t|^2} \leqslant \int_{0}^{+\infty} \frac{\de t}{t^2+(x^2+y^2)} = \frac{\pi}{2\sqrt{x^2+y^2}} = \frac{\pi}{2\operatorname{dist}(z,\mathbb{R}_{\leqslant 0})}.\qedhere
\]
\end{proof}
		
\begin{lemma}
\label{lem_tilde_S_n}
Let $f(z)$ and $g(z)$ be the functions defined in Definition \ref{def_f_g}.
For any sufficiently large $n \in q!\mathbb{N}$ and any $\mu \in (0,q)$, we have
\[
S_n=n^{O(1)} \cdot \widetilde{S}_n,
\]
where
\begin{equation}
\label{def_tilde_S_n}
\widetilde{S}_n:=\frac{1}{2\pi i}\int_{\mu-i\infty}^{\mu+i\infty}\sin^k(n\pi z)\cdot\cot_k(n\pi z)\cdot e^{nf(z)}\cdot g_n(z)\,\de z,
\end{equation}
and $g_n(z)$ is a holomorphic function on $\mathbb{C} \setminus \left( (-\infty,0] \cup [q,+\infty) \right)$ such that 
\begin{equation}
\label{g_n=g}
g_n(z) = g(z) \left( 1+O\left((\varepsilon_0n)^{-1}\right) \right) \quad\text{uniformly on~} D_{\varepsilon_0},
\end{equation} 
with
\begin{equation}\label{def_D_eps}
D_{\varepsilon_0} := \left\{ z \in \mathbb{C} ~\big|~ \operatorname{dist}(z,\mathbb{R}_{\leqslant 0} \cup \mathbb{R}_{\geqslant q}) > \varepsilon_0 \right\}
\end{equation}
for any preassigned $\varepsilon_0 \in (0,1)$.
The implicit constants depend only on $k,q,r$. 
\end{lemma}
		
\begin{proof}
Taking $M=n\mu$ in Lemma \ref{S_n=int_cot_R} and changing the variable $z$ to $nz$, we obtain
\[
S_n = \frac{n\pi^{k-1}i}{2}\int_{\mu-i\infty}^{\mu+i\infty} \cot_k(n\pi z) R_n(nz)\,\de z.
\]
			
We rewrite each Pochhammer symbol in $R_n(nz)$ (see Definition \ref{def_R_n(t)}) as a ratio of Gamma functions using the
identities
\[
(nz-qn)_{qn}
=
\frac{\Gamma(nz)}{\Gamma(nz-qn)},
\]
\[
(nz+rn+1)_{qn}
=
\frac{\Gamma(nz+rn+qn+1)}
     {\Gamma(nz+rn+1)},
\]
and
\[
(qnz)_{rqn+1}
=
\frac{\Gamma(qnz+rqn+1)}{\Gamma(qnz)}.
\]
We then apply $\Gamma(w+1)=w\Gamma(w)$ to remove the shifts
by $1$. 
These identities are understood meromorphically.
We obtain
\begin{align*}
R_n(nz) =&~ q^{2kqn} \prod_{p \mid q \atop p \text{~prime}} p^{2kqn/(p-1)} \cdot r \cdot (qn)^{1-\delta_k-2k} \\
&\times \frac{(2z+r)^{1-\delta_k}}{z+r} \cdot \left( \frac{z+r+q}{z+r} \right)^k \\
&\times \left(\frac{\Gamma(nz)}{\Gamma(nz-qn)\Gamma(qn)} \cdot \frac{\Gamma(nz+rn+qn)}{\Gamma(nz+rn)\Gamma(qn)}\right)^k \cdot \frac{\Gamma(qnz)\Gamma(rqn)}{\Gamma(qnz+rqn)}.
\end{align*}
Using Euler's reflection formula for the Gamma function
\[
\Gamma(nz-qn)\Gamma(qn-nz) = -\frac{\pi}{(nz-qn)\sin(\pi nz-qn\pi)},
\]
and the fact that $n$ is even, we obtain
\begin{align*}
R_n(nz) =&~ \pi^{-k} \cdot q^{2kqn} \prod_{p \mid q \atop p \text{~prime}} p^{2kqn/(p-1)} \cdot r \cdot q^{1-\delta_k-2k} n^{1-\delta_k -k} \\
&\times \frac{(2z+r)^{1-\delta_k}}{z+r} \cdot \left( \frac{(q-z)(z+r+q)}{z+r} \right)^k \\
&\times \sin^{k}(\pi n z) \cdot \left(\frac{\Gamma(nz)\Gamma(qn-nz)}{\Gamma(qn)} \cdot \frac{\Gamma(nz+rn+qn)}{\Gamma(nz+rn)\Gamma(qn)}\right)^k \cdot \frac{\Gamma(qnz)\Gamma(rqn)}{\Gamma(qnz+rqn)}.
\end{align*}
			
By Lemma \ref{lem_Stirling}, we have Stirling's formula
\[
\log\Gamma(w) = w\log w - w -\frac{\log w}{2} +\frac{\log(2\pi)}{2} + O\left( \operatorname{dist}(w,\mathbb{R}_{\leqslant 0})^{-1}\right), \quad w \in \mathbb{C}\setminus (-\infty,0].
\]
A straightforward computation using Stirling's formula shows that
\[
R_n(nz) = \frac{1}{\pi^{k}(qn)^{k-1+\delta_k}} \sqrt{\frac{2r\pi}{qn}}  \cdot \sin^{k}(n\pi z) e^{nf(z)}g(z)\left(1+O\left((\varepsilon_0 n)^{-1}\right)\right),\quad z \in D_{\varepsilon_0},
\]
where the domain $D_{\varepsilon_0}$ is defined by \eqref{def_D_eps} and the implicit constant depends only on $k,q,r$.
Let us define the function $g_n(z)$ on $\mathbb{C} \setminus \big((-\infty,0] \cup [q,+\infty)\big)$ by
\begin{equation}
\label{def_g_n}
R_n(nz) =: \frac{1}{\pi^{k}(qn)^{k-1+\delta_k}} \sqrt{\frac{2r\pi}{qn}}  \cdot \sin^{k}(n\pi z) e^{nf(z)}g_n(z).
\end{equation}
Then, we have
\[
g_n(z) = g(z)\left(1+O\left((\varepsilon_0 n)^{-1}\right)\right) \quad\text{uniformly for~} z \in D_{\varepsilon_0},
\]
and
\begin{align*}
S_n =&~ \frac{i}{(qn)^{k-1+\delta_k}}\sqrt{\frac{rn}{2q\pi}} \cdot \int_{\mu-i\infty}^{\mu+i\infty} \sin^k(n\pi z)\cdot\cot_k(n\pi z)\cdot e^{nf(z)}\cdot g_n(z)\,\de z \\
=&~ n^{O(1)} \cdot \widetilde{S}_n.
\end{align*}
The proof of Lemma \ref{lem_tilde_S_n} is complete.
\end{proof}
		
\bigskip
		
\begin{lemma}
\label{lem_asym_f_g}
As $z \in \mathbb{C} \setminus \mathbb{R}$ and $|z| \to +\infty$, we have $g(z) = O(1)$ and
\[
f(z) = \operatorname{sgn}(\operatorname{Im} z) \cdot k\pi i z -(r-2k)q\log z  + O(1),
\]
where the implicit constants depend only on $k,q,r$.
\end{lemma}
		
\begin{proof}
Clearly, by \eqref{def_g} we have
\[
|g(z)| \leqslant \frac{\max\{1,|2z+r|\}}{\sqrt{|z|}\sqrt{|z+r|}}\left(\frac{\sqrt{|q-z|}\sqrt{|z+r+q|}}{\sqrt{|z|}\sqrt{|z+r|}}\right)^{k} = O(1).
\]
On the other hand, the claimed asymptotic behavior of $f(z)$ follows by substituting the estimates below into \eqref{def_f}:
\begin{align*}
&\log(z+r+q) = \log z + O\left( |z|^{-1} \right), \\
&\log(q-z) = \log z - \operatorname{sgn}(\operatorname{Im} z) \cdot \pi i +  O\left( |z|^{-1} \right), \\
&\log(z+r) = \log z + O\left( |z|^{-1} \right).
\end{align*}
\end{proof}
		
\begin{lemma}
\label{lem_S=sum_J}
For any $n \in q!\mathbb{N}$, $\mu \in (0,q)$, and $\lambda \in (-k,k)$, the integral 
\begin{equation}
\label{def_J}
J_{n,\lambda} := \frac{1}{2\pi i}\int_{\mu-i\infty}^{\mu+i\infty} e^{n(f(z)-\lambda \pi i z)} g_n(z) \,\mathrm{d}z  
\end{equation}
is absolutely convergent. 
Moreover, we have
\begin{equation}
\label{S=sum_J}
\widetilde{S}_n = \sum_{0 \leqslant l \leqslant k-2 \atop l \equiv k \pmod{2}} c_l \operatorname{Re}(J_{n,l}),
\end{equation}
where the constants $c_l$ \textup{(}$0 \leqslant l \leqslant k-2$, $l \equiv k \pmod{2}$\textup{)} depend only on $k$ and $c_{k-2} \neq 0$.
\end{lemma}
		
\begin{proof}
As $z= \mu + it$ and $|t| \to +\infty$, we deduce from Lemma \ref{lem_asym_f_g} and Equation \eqref{g_n=g} that $g_n(z) = O(1)$ and
\[
\operatorname{Re}(f(z)-\lambda \pi i z) = -k\pi|t| + \lambda \pi t + O(\log |t|).
\]
Since $|\lambda| < k$, the integrand function in $J_{n,\lambda}$ decays exponentially at both $\mu \pm i\infty$. 
Therefore, the integral $J_{n,\lambda}$ converges absolutely. 
Then, Equation \eqref{def_tilde_S_n} and Lemma \ref{property_of_cot_k} (3) imply that
\begin{align*}
\widetilde{S}_{n} &= \frac{1}{2\pi i}\int_{\mu - i\infty}^{\mu+i\infty} \left( \sum_{0 \leqslant l \leqslant k-2 \atop l \equiv k \pmod{2}} c_l \cos(l n\pi z) \right)e^{nf(z)}g_n(z)\,\mathrm{d}z \\
&= \frac{1}{2}\sum_{0 \leqslant l \leqslant k-2 \atop l \equiv k \pmod{2}} c_l\left(J_{n,l}+J_{n,-l}\right).
\end{align*}
By \eqref{def_f}, we have $f(\overline{z})=\overline{f(z)}$.
Then, by \eqref{def_g_n}, we have $g_n(\overline{z})=\overline{g_n(z)}$. 
It follows that $\overline{J_{n,l}} = J_{n,-l}$, and hence \eqref{S=sum_J} holds.
\end{proof}
		
\begin{lemma}
\label{lem_general_L}
Let $n \in q!\mathbb{N}$, $\mu \in (0,q)$, and $\lambda \in (-k,k)$.
\begin{itemize}
\item[\textup{(1)}] The contour of integration $\operatorname{Re}z=\mu$ in the integral \eqref{def_J} can be replaced by any other contour $\mathcal{L}$ parametrized by
\[
z(t)=
\begin{cases}
\mu+i(t-\mu) &\text{if~} t\in(-\infty,\mu],\\
t+iy(t) &\text{if~} t\in[\mu,\mu^*],\\
t+iy(\mu^*) &\text{if~} t\in[\mu^*,+\infty),
\end{cases}
\]
where $\mu^{*}$ is any real number such that $\mu^*>\mu$ and $y:[\mu,\mu^*]\to\mathbb{R}$ is continuous and non-decreasing, is $C^1$ except at finitely many points, and satisfies $y(\mu)=0$ and $y(t)>0$ for $t\in(\mu,\mu^*]$. 
See Figure \textup{\ref{fig:general_L}}.
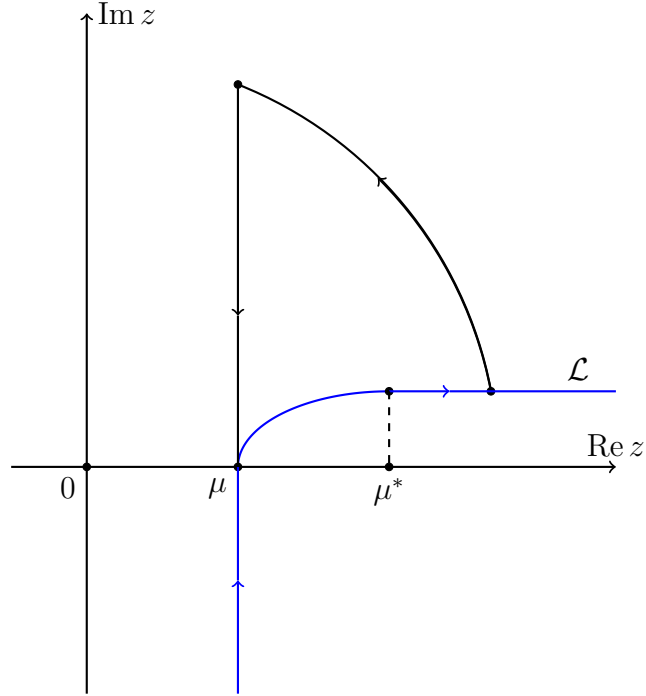
\begin{figure}[h]
\centering
\begin{tikzpicture}
\draw [thick] [->] (-1,0)--(7,0);
\draw [thick] [->] (0,-3)--(0,6);
\draw[blue, domain=90:180, smooth, variable=\t,thick]
plot ({2*cos(\t)+4}, {1*sin(\t)});
\draw[fill] (0,0) circle [radius=0.05];
\draw[fill] (4,0) circle [radius=0.05];
\draw[fill] (2,0) circle [radius=0.05];
\draw[fill] (4,1) circle [radius=0.05];
\draw[fill] (5.3471,1) circle [radius=0.05];
\draw[fill] (2,5.05880207) circle [radius=0.05];
\draw[thick] [->] (5.3471,1) arc [radius=5.4441, start angle=10.8505, end angle= 45];
\draw[thick] (2,5.05880207) arc [radius=5.4441, start angle=68.428609718, end angle= 10.8505];
\draw [thick,black] [->] (2,5.05880207)--(2,2);
\draw [thick,blue] [->] (2,-3)--(2,-1.5);
\draw [thick,blue]  (2,-1.5)--(2,0);
\draw [thick,black]  (2,0)--(2,2);
\draw [thick,blue] [->] (4,1)--(4.8,1);
\draw [thick,blue]  (4.8,1)--(7,1);
\draw [thick,dashed]  (4,1)--(4,0);
\node [below] at (4,0) {$\mu^*$};
\node [below left] at (0,0) {$0$};
\node [below left] at (2,0) {$\mu$};
\node [above] at (7,0) {$\operatorname{Re}z$};
\node [right] at (0,6) {$\operatorname{Im}z$};
\node [above] at (6.5,1) {$\mathcal{L}$};
\end{tikzpicture}
\caption{contour $\mathcal{L}$ (blue).}
\label{fig:general_L}
\end{figure}
				
\item[\textup{(2)}] For any sufficiently large real number $T$, we have
\[
\int_{|t|>T} | e^{n(f(z(t))-\lambda\pi iz(t))}g_n(z(t))z'(t)|\,\de t=O\left(T^{-n}\right),
\]
where the implicit constant depends only on $k,q,r$.
\end{itemize}
\end{lemma}
		
\begin{proof}
Take a small $\varepsilon_0 > 0$ such that both contours $\operatorname{Re}z=\mu$ and $\mathcal{L}$ lie in the domain $D_{\varepsilon_0}$ defined by \eqref{def_D_eps}. 
As $z \in D_{\varepsilon_0}$ and $|z| \to +\infty$, by Lemma \ref{lem_asym_f_g}, Equation \eqref{g_n=g}, and the fact $|\lambda|<k$, we have
\begin{align}
\operatorname{Re}(f(z)-\lambda \pi i z) &= -k\pi |\operatorname{Im}z| +\lambda \pi \operatorname{Im} z -(r-2k)q\log|z| +O(1) \notag\\
&\leqslant -(r-2k)q\log|z| +O(1), \label{731}\\
|g_n(z)| &= O(1), \label{732}
\end{align}
where the implicit constants depend only on $k,q,r,\varepsilon_0$.
			
We first prove part (1). 
For any sufficiently large $T$, we join the point $z(T)$ of the new contour $\mathcal{L}$ and a point of the original contour $\operatorname{Re}(z)=\mu$ by an arc $\gamma_T$ of radius $|z(T)|$ with center at the origin. 
By \eqref{731} and \eqref{732}, we have
\[
\int_{\gamma_T} \left|e^{n(f(z)-\lambda\pi iz)}g_n(z)\right|\,|\de z| \leqslant \frac{e^{O(n)}}{T^{(r-2k)qn-1}} \to 0 \quad\text{as~} T \to +\infty.
\]
Therefore, we can deform the contour from $\operatorname{Re}z=\mu$ to $\mathcal{L}$ without changing the value of the integral $J_{n,\lambda}$.
			
Now we prove part (2).
By \eqref{731} and \eqref{732}, we have
\[
\left| e^{n(f(z(t))-\lambda\pi iz(t))}g_n(z(t))z'(t)\right| \leqslant \frac{e^{O(n)}}{|t|^{(r-2k)qn}} \quad\text{as~} t \to \pm\infty,
\]
and hence,
\[
\int_{|t|>T}  \left|e^{n(f(z(t))-\lambda\pi iz(t))}g_n(z(t))z'(t)\right|\de t\leqslant \frac{e^{O(n)}}{T^{(r-2k)qn-1}} =O(T^{-n}).\qedhere
\]
\end{proof}
		
\bigskip
		
\section{Solutions of the equation $h(z)=\lambda\pi i$}
\label{Sect_h(z)}
		
In view of Lemma \ref{lem_S=sum_J}, our aim is to estimate $J_{n,l}$ using the saddle-point method. 
The saddle points of the function $f(z) - l \pi i z$ are solutions of the equation $f'(z)=l \pi i$. 
Note that
\[
f'(z) = (q+k)(\log z - \log(z+r)) + k(\log(z+r+q) - \log(q-z)). 
\]
By the substitution $z=r(w-1)/2$, we obtain
\begin{equation}\label{why_h}
f'\left(\frac{r(w-1)}{2}\right)=(a+b)\big(\log(w-1)-\log(w+1)\big)+b\big(\log(1+s+w)-\log(1+s-w)\big),
\end{equation}
where $a=q$, $b=k$, and $s=2q/r$. 
Note that the condition $r>2k$ converts to $a>sb$. 
In the following, we consider functions of the form \eqref{why_h} in a slightly more general context. 
		
\begin{definition}
\label{def_h}
Fix $a,b,s\in\mathbb{R}_{>0}$ with $a>sb$. 
We define the following holomorphic function on $\mathbb{C}\setminus \big( (-\infty,1] \cup [1+s,+\infty) \big)$.
\begin{equation*}
h(z):=(a+b)\big(\log(z-1)-\log(z+1)\big)+b\big(\log(1+s+z)-\log(1+s-z)\big).
\end{equation*}
\end{definition}
		
In this section, we focus on studying the solutions of the equation
\begin{equation}
\label{h=lambda_pi_i}
h(z)=\lambda\pi i,
\end{equation}
where $\lambda\in\mathbb{R}$ is a fixed parameter. 
Equation \eqref{h=lambda_pi_i} has been studied by Zudilin in \cite{Zud2002} under some additional assumptions. 
We remove all unnecessary assumptions and simplify Zudilin's arguments.
		
\subsection{Real part of $h(z)$}
		
Note that any solution $z$ of Equation \eqref{h=lambda_pi_i} satisfies $\operatorname{Re}(h(z)) = 0$. 
Write $z=x+iy$, where $x,y \in \mathbb{R}$. 
Then
\begin{equation*}
\operatorname{Re}(h(z))=\frac{a+b}{2}\log\frac{(x-1)^2+y^2}{(x+1)^2+y^2}+\frac{b}{2}\log\frac{(x+1+s)^2+y^2}{(x-1-s)^2+y^2}.
\end{equation*}
It is convenient to define the following.
\begin{definition}
\label{def_H(x,y)}
Fix $a,b,s\in\mathbb{R}_{>0}$ with $a>sb$. 
Define the continuous function $H: \mathbb{R}^2 \rightarrow \mathbb{R} \cup\{\pm \infty\}$ by
\[
H(x,y) := \frac{a+b}{2}\log\frac{(x-1)^2+y^2}{(x+1)^2+y^2}+\frac{b}{2}\log\frac{(x+1+s)^2+y^2}{(x-1-s)^2+y^2}.
\]
\end{definition}
Here, the extended real line $\mathbb{R} \cup \{\pm \infty\}$ is equipped with the order topology. 
Note that the function $H(x,y)$ is an extension of $\operatorname{Re}(h(x+iy))$. 
In this subsection, we study the solutions of the equation $H(x,y)=0$. 
Since $H(x,y)$ satisfies
\begin{equation}
\label{sym_of_H_1}
H(-x,y) = -H(x,y), \quad H(x,-y)=H(x,y)
\end{equation}
for any $(x,y) \in \mathbb{R}^2$ and
\begin{equation}
\label{sym_of_H_2}
H(0,y) = 0 \quad\text{for any~} y \in \mathbb{R},
\end{equation}
it suffices to consider $x > 0$ and $y \geqslant 0$.
		
\bigskip
		
\begin{lemma}
\label{H(x,0)}
There exist a unique $\eta_0 \in (1,1+s)$ and a unique $\eta_1 \in (1+s,+\infty)$ such that
\begin{equation}
\label{sign_of_H(x,0)}
H(x,0)
\begin{cases}
<0&\text{if}\; x\in(0,\eta_0)\cup(\eta_1,+\infty),\\
=0&\text{if}\; x=\eta_0 \text{~or~} x=\eta_1,\\
>0&\text{if}\; x\in(\eta_0,\eta_1).
\end{cases}
\end{equation}
Moreover, we have
\begin{equation}
\label{partial_H(eta_0)}
\frac{\partial H}{\partial x}(\eta_0,0)>0 \quad\text{and}\quad \frac{\partial H}{\partial x}(\eta_1,0)<0.
\end{equation}
\end{lemma}
		
\begin{proof}
We start with the identity
\[
H(x,0)=
\begin{cases}
(a+b)\log\frac{1-x}{1+x}+b\log\frac{1+s+x}{1+s-x} &\text{if}\; 0<x<1,\\
(a+b)\log\frac{x-1}{x+1}+b\log\frac{1+s+x}{1+s-x} &\text{if}\; 1<x<1+s,\\
(a+b)\log\frac{x-1}{x+1}+b\log\frac{x+1+s}{x-1-s} &\text{if}\; x>1+s.
\end{cases}
\]
By a straightforward computation, we obtain
\[
\frac{\partial H}{\partial x}(x,0)
\begin{cases}
<0 &\text{if}\; 0<x<1,\\
>0 &\text{if}\; 1<x<1+s,\\
<0 &\text{if}\; 1+s<x<\sqrt{\frac{(a+b)(1+s)^2-b(1+s)}{a-sb}},\\
>0 &\text{if}\; x>\sqrt{\frac{(a+b)(1+s)^2-b(1+s)}{a-sb}}.
\end{cases}
\]
Note that
\[
H(0,0)=0, \quad H(1,0)=-\infty, \quad H(1+s,0)=+\infty, \quad\text{and}~\lim_{x\to+\infty} H(x,0)=0.
\]
Therefore, the equation $H(x,0)=0$ has one solution $\eta_0 \in (1,1+s)$ and another solution $\eta_1 \in (1+s,+\infty)$, satisfying \eqref{sign_of_H(x,0)} and \eqref{partial_H(eta_0)}.
\end{proof}
		
\begin{lemma}
\label{H(x,y)}
Let $\eta_0$ and $\eta_1$ be the real numbers defined in Lemma \ref{H(x,0)}. Then, there exists a $C^1$ function $Y_0 \colon (\eta_0,\eta_1) \rightarrow \mathbb{R}_{>0}$ such that
\begin{equation}
\label{sign_of_H(x,y)}
H(x,y)
\begin{cases}
<0 &\text{if}\; x \in (0,\eta_0] \cup [\eta_1,+\infty) \;\text{and}\; y > 0, \\
<0 &\text{if}\; x \in (\eta_0,\eta_1) \;\text{and}\; y > Y_0(x), \\
=0 &\text{if}\; x \in (\eta_0,\eta_1) \;\text{and}\; y = Y_0(x), \\
>0  &\text{if}\; x \in (\eta_0,\eta_1) \;\text{and}\; 0< y < Y_0(x). \\
\end{cases}
\end{equation}
Moreover, we have 
\begin{equation}
\label{Y_0_is_continuous_at_end_points}
\lim_{x \to \eta_0^+} Y_0(x)=0, \quad \lim_{x \to \eta_1^-} Y_0(x)=0, 
\end{equation}
and
\begin{equation}
\label{H_partial_y<0_on_C}
\frac{\partial H}{\partial y}(x,Y_0(x)) < 0 \quad \text{for any~} x\in (\eta_0,\eta_1).
\end{equation}
\end{lemma}
		
\begin{proof}
A straightforward calculation shows that
\[
\frac{\partial H}{\partial y}(x,y) =\frac{4xy\cdot\Big((a-sb)y^4 + c_1(x)y^2 +c_2(x)\Big)}{|z-1|^2|z+1|^2|z-1-s|^2|z+1+s|^2},
\]
where $z=x+iy$ and 
\begin{align*}
c_1(x) &= 2(a-sb)x^2+2(1+s)(a+sa+sb),\\
c_2(x) &= (a+b)\big(x^2-(1+s)^2\big)^2 - b(1+s)(x^2-1)^2.
\end{align*}
Fix $x>0$. Since $a-sb>0$ and $c_1(x)>0$, the behavior of the function $y \mapsto H(x,y)$ on $[0,+\infty)$ has two possibilities, depending on whether $c_2(x) \geqslant 0$ or $c_2(x) < 0$, as follows.
\begin{enumerate}
\item[(P1)] The function $y \mapsto H(x,y)$ increases strictly on the interval $[0,+\infty)$; or
\item[(P2)] There exists a real number $\xi(x)>0$ such that the function $y \mapsto H(x,y)$ decreases strictly on the interval $[0,\xi(x)]$ and increases strictly on the interval $[\xi(x),+\infty)$.
\end{enumerate}
			
If $x \in (0,\eta_0] \cup [\eta_1,+\infty)$, then we have $H(x,0) \leqslant 0$ (by \eqref{sign_of_H(x,0)}) and $\lim\limits_{y\to+\infty}H(x,y)=0$. 
No matter which of (P1) or (P2) occurs, we always have $H(x,y)<0$ for any $y>0$.
			
If $x \in (\eta_0,\eta_1)$, then we have $H(x,0) > 0$ (by \eqref{sign_of_H(x,0)}) and $\lim\limits_{y\to+\infty}H(x,y)=0$. 
In this case, only (P2) can occur. 
We deduce that there exists a real number $Y_0(x) \in (0,\xi(x))$ such that 
\[
H(x,y)
\begin{cases}
>0 &\text{if}\; 0<y<Y_0(x), \\
=0 &\text{if}\; y=Y_0(x), \\
<0 &\text{if}\; y>Y_0(x),
\end{cases}
\quad\text{and}\quad \frac{\partial H}{\partial y}(x,Y_0(x)) < 0.
\]
			
In summary, there exists a function $Y_0 \colon (\eta_0,\eta_1) \rightarrow \mathbb{R}_{>0}$ such that Equations \eqref{sign_of_H(x,y)} and \eqref{H_partial_y<0_on_C} hold. 
Then, the implicit function theorem, together with Equations \eqref{sign_of_H(x,y)} and \eqref{H_partial_y<0_on_C}, implies that the function $Y_0$ is $C^1$ on $(\eta_0,\eta_1)$. 
Finally, by Lemma \ref{H(x,0)} and the implicit function theorem, there exist a small $\varepsilon_0>0$ and two $C^1$ functions $X_0,X_1 \colon (-\varepsilon_0,\varepsilon_0) \rightarrow \mathbb{R}_{>0}$ such that
\begin{equation}
\label{X_j}
X_j(0)=\eta_j, \quad H(X_j(y),y)=0 \text{~for any~} y \in (-\varepsilon_0,\varepsilon_0), \quad j=0,1.
\end{equation}
Comparing \eqref{X_j} with \eqref{sign_of_H(x,y)}, we obtain \eqref{Y_0_is_continuous_at_end_points}.
\end{proof}
		
By Lemmas \ref{H(x,0)} and \ref{H(x,y)}, together with Equations \eqref{sym_of_H_1} and \eqref{sym_of_H_2}, we have completely determined the sign of $H(x,y)$. 
See Figure \ref{fig:sign_of_H(x,y)}.
		
\begin{figure}[h]
\centering
\begin{tikzpicture}
\draw [thick] [->] (-7,0)--(7,0);
\draw [red, thick] [->] (0,-3)--(0,3);
\draw[red, domain=90:270, smooth, variable=\t,thick]
plot ({2*cos(\t)+4}, {1*sin(\t)});
\draw[red, domain=-90:90, smooth, variable=\t,thick]
plot ({1.2*cos(\t)+4}, {1*sin(\t)});
\draw[red, domain=-90:90, smooth, variable=\t,thick]
plot ({2*cos(\t)-4}, {1*sin(\t)});
\draw[red, domain=90:270, smooth, variable=\t,thick]
plot ({1.2*cos(\t)-4}, {1*sin(\t)});
\draw[fill] (1,0) circle [radius=0.05];
\draw[fill] (-1,0) circle [radius=0.05];
\draw[fill] (3,0) circle [radius=0.05];
\draw[fill] (-3,0) circle [radius=0.05];
\node [below] at (1,0) {$1$};
\node [below] at (-1,0) {$-1$};
\node [below] at (3,0) {$1+s$};
\node [below] at (-3,0) {$-1-s$};
\draw [thick] [->] (1.7,-0.5)--(1.95,-0.05);
\node [below] at (1.7,-0.5) {$\eta_0$};
\draw [thick] [->] (-1.7,-0.5)--(-1.95,-0.05);
\node [below] at (-1.7,-0.5) {$-\eta_0$};
\draw [thick] [->] (-5.5,-0.5)--(-5.25,-0.05);
\node [below] at (-5.5,-0.5) {$-\eta_1$};
\draw [thick] [->] (5.5,-0.5)--(5.25,-0.05);
\node [below] at (5.5,-0.5) {$\eta_1$};
\node [above] at (7,0) {$x$};
\node [right] at (0,3) {$y$};
\node [above] at (3.7,0) {$H(x,y)>0$};
\node [above] at (5,2) {$H(x,y)<0$};
\node [above] at (-3.7,0) {$H(x,y)<0$};
\node [above] at (-5,2) {$H(x,y)>0$};
\end{tikzpicture}
\caption{sign of $H(x,y)$.}
\label{fig:sign_of_H(x,y)}
\end{figure}
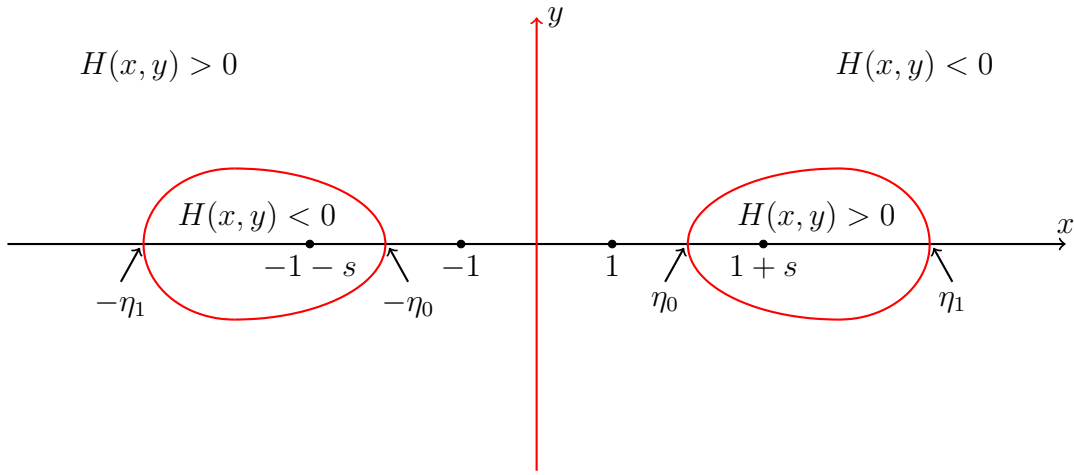
		
\subsection{Imaginary part of $h(z)$}
		
By Definition \ref{def_h}, we have
\begin{equation}
\label{Im(h)}
\operatorname{Im}(h(z))=(a+b)\big(\arg(z-1)-\arg(z+1)\big)+b\big(\arg(1+s+z)-\arg(1+s-z)\big)
\end{equation}
for any $z \in \mathbb{C}\setminus \big( (-\infty,1] \cup [1+s,+\infty) \big)$, where each $\arg(\cdot)$ takes values in $(-\pi,\pi)$. 
		
\begin{lemma}
\label{im(h(z))}
Let $Y_0$ be the function defined in Lemma \ref{H(x,y)}. Then, the function $x \mapsto \operatorname{Im}(h(x+iY_0(x)))$ increases strictly on the interval $(\eta_0,\eta_1)$. Moreover, we have
\begin{equation}
\label{end_points_of_im(h)_on_C}
\lim_{x\to\eta_0^+} \operatorname{Im}(h(x+iY_0(x))) = 0 \quad\text{and}\quad \lim_{x\to\eta_1^-} \operatorname{Im}(h(x+iY_0(x))) = b\pi.
\end{equation}
\end{lemma}
		
\begin{proof}
Write $z=x+iy$ and $h(z)=u(x,y)+iv(x,y)$. By the Cauchy--Riemann equations, we have
\[
\frac{\partial u}{\partial x} = \frac{\partial v}{\partial y} \quad\text{and}\quad \frac{\partial u}{\partial y} = - \frac{\partial v}{\partial x}.
\]
By Equation \eqref{sign_of_H(x,y)}, we have $u(x,Y_0(x)) = 0$ for any $x \in (\eta_0,\eta_1)$, and hence
\[
\frac{\partial u}{\partial x}(x,Y_0(x)) + \frac{\partial u}{\partial y}(x,Y_0(x)) \cdot Y_0^{\prime}(x) = 0 \quad\text{for any~} x \in (\eta_0,\eta_1).
\]
Therefore, we have
\begin{align}
&~\frac{\de}{\de x} \operatorname{Im}(h(x+iY_0(x))) = \frac{\de}{\de x} v(x,Y_0(x)) = \frac{\partial v}{\partial x} + \frac{\partial v}{\partial y} \cdot Y_0^{\prime} \notag\\
=&~ -\frac{\partial u}{\partial y} + \frac{\partial u}{\partial x} \cdot Y_0^{\prime} 
= -\frac{\partial u}{\partial y} + \left( -\frac{\partial u}{\partial y}\cdot Y_0^{\prime} \right) \cdot Y_0^{\prime} = -(1+(Y_0^{\prime})^2) \cdot \frac{\partial u}{\partial y}. \label{651}
\end{align}
By Equations \eqref{651} and \eqref{H_partial_y<0_on_C}, we obtain 
\[
\frac{\de}{\de x} \operatorname{Im}(h(x+iY_0(x))) > 0 \quad\text{for any~} x \in (\eta_0,\eta_1).
\]
Finally, the limits in \eqref{end_points_of_im(h)_on_C} follow from Equations \eqref{Y_0_is_continuous_at_end_points}, \eqref{Im(h)}, and the fact $1<\eta_0<1+s<\eta_1$.
\end{proof}
				
\begin{lemma}
\label{psi(y)}
The function $y \mapsto \operatorname{Im}(h(iy))$ decreases strictly on the interval $(0,+\infty)$. 
Moreover, we have
\begin{equation}
\label{661}
\lim_{y\to0^+} \operatorname{Im}(h(iy))=(a+b)\pi \quad\text{and}\quad \lim_{y\to+\infty} \operatorname{Im}(h(iy))=b\pi.
\end{equation}
\end{lemma}
		
\begin{proof}
For any $y \in (0,+\infty)$, we have
\begin{equation}
\label{662}
\operatorname{Im}(h(iy)) = 2(a+b)\arctan \frac{1}{y} + 2b \arctan \frac{y}{1+s}
\end{equation}
and
\[
\frac{\de}{\de y} \operatorname{Im}(h(iy)) = -2\frac{(a-sb)y^2+(1+s)(a+sa+sb)}{(1+y^2)((1+s)^2+y^2)}<0.
\]
The limits in \eqref{661} follow from \eqref{662}.
\end{proof}

\subsection{Distribution of the solutions}
		
In this subsection, we determine the solutions of the equation $h(z)=\lambda \pi i$ for any fixed $\lambda \in \mathbb{R}$. We consider not only solutions in the domain $\mathbb{C}\setminus \big( (-\infty,1] \cup [1+s,+\infty) \big)$ of $h(z)$, but also solutions on the upper or lower bank of the cuts $(-\infty,1]$ and $[1+s,+\infty)$.
		
\begin{lemma}
\label{roots_of_h}
Fix $\lambda\in\mathbb{R}$. 
Then, the equation
\begin{equation}
\label{h=lambda}
h(z)=\lambda\pi i 
\end{equation}
has the following solutions\textup{:}
\begin{enumerate}
\item[\textup{(1)}] For $\lambda = 0$, there is a pair of real solutions $-\eta_0 \pm i0$ and a solution $\eta_0$, where $+(-)$ in $\pm i0$ corresponds to the upper \textup{(}lower\textup{)} bank of the cut $(-\infty,1]$\textup{;}
				
\item[\textup{(2)}] For $\lambda=\pm b$, there is a pair of real solutions $-\eta_1\pm i0$ and $\eta_1\pm i0$\textup{;}
%where $+(-)$ in $\pm i0$ coincides with the sign of $\lambda$ and corresponds to the upper \textup{(}lower\textup{)} banks of the cuts $(-\infty, 1]$ and $[1+s, +\infty)$\textup{;}
				
\item[\textup{(3)}] For $\lambda=\pm(a+b)$, there is a real solution $\pm i0$\textup{;}
%where $+(-)$ in $\pm i0$ coincides with the sign of $\lambda$ and corresponds to the upper \textup{(}lower\textup{)} bank of the cut $(-\infty,1]$\textup{;}
				
\item[\textup{(4)}] For real $\lambda$ satisfying $b<|\lambda|<a+b$, there is a purely imaginary solution\textup{;}
				
\item[\textup{(5)}] For real $\lambda$ satisfying $0<|\lambda|<b$, there is a pair of non-real solutions symmetric with respect to the line $\mathrm{Re}(z)=0$\textup{;}
				
\item[\textup{(6)}] For real $\lambda$ satisfying $|\lambda|>a+b$, there is no solution.
\end{enumerate}		
All solutions of Equation \eqref{h=lambda} appear in the list above. 
All solutions of Equation \eqref{h=lambda} corresponding to positive $\lambda$ lie in the upper half-plane $\operatorname{Im}(z)>0$ or on an upper bank of a cut. 
All solutions of Equation \eqref{h=lambda} corresponding to negative $\lambda$ lie in the lower half-plane $\operatorname{Im}(z)<0$ or on a lower bank of a cut.
\end{lemma}
		
\begin{proof}
Any solution $z$ of Equation \eqref{h=lambda} satisfies $\operatorname{Re}(h(z)) = 0$. By Lemmas \ref{H(x,0)} and \ref{H(x,y)}, together with Equations \eqref{sym_of_H_1} and \eqref{sym_of_H_2}, the only candidates are (see Figure \ref{fig:sign_of_H(x,y)})
\begin{itemize}
\item $z = \pm x \pm iY_0(x)$ for some $x \in (\eta_0,\eta_1)$;
\item $z= \pm iy$ for some $y>0$;
\item $z=\eta_0$, $\eta_1 \pm i0$, $\pm i0$, $-\eta_0 \pm i0$, $-\eta_1 \pm i0$.
\end{itemize}
Then, considering Lemmas \ref{im(h(z))}, \ref{psi(y)}, and the symmetry of $h(z)$, it is straightforward to verify that (1)--(6) exhaust all solutions of Equation \eqref{h=lambda}. 
\end{proof}
		
\subsection{Further properties of the function $Y_0(x)$}
		
In this subsection, we establish some further properties of the function $Y_0(x)$ defined in Lemma \ref{H(x,y)}. 
These properties will be used in the next section.
		
\begin{lemma}
\label{y'(x)=0}
If $x_0 \in (\eta_0,\eta_1)$ satisfies $Y_0^{\prime}(x_0) = 0$, then $x_0$ is the unique solution of the equation $Y_0(x)=Y_0(x_0)$ in the interval $(\eta_0,\eta_1)$.
\end{lemma}
		
\begin{proof}
Suppose that $x_0 \in (\eta_0,\eta_1)$ satisfies $Y_0^{\prime}(x_0) = 0$; let $y_0 = Y_0(x_0)$.
By \eqref{sign_of_H(x,y)}, we have $H(x,Y_0(x)) = 0$ for any $x \in (\eta_0,\eta_1)$, and hence
\begin{equation}
\label{681}
\frac{\partial H}{\partial x}(x_0,y_0) + \frac{\partial H}{\partial y}(x_0,y_0) \cdot Y_0^{\prime}(x_0) = 0 \Longrightarrow \frac{\partial H}{\partial x}(x_0,y_0) = 0.
\end{equation}
By a straightforward calculation using Definition \ref{def_H(x,y)}, we have
\[
\frac{\partial H}{\partial x}(x,y_0) = \frac{Q(x^2)}{|z-1|^2|z+1|^2|z-1-s|^2|z+1+s|^2},
\]
where $z=x+iy_0$ and $Q(t)$ is a polynomial of the form $Q(t) = 2(a-sb)t^3 + \cdots$. 
Note that $\deg Q = 3$.
			
Note that $H(0,y_0) = 0$, $H(x_0,y_0) =0$, and $\lim\limits_{x\to+\infty} H(x,y_0) = 0$. 
If there exists a real number $x_1 \in (\eta_0,\eta_1) \setminus \{ x_0 \}$ such that $Y_0(x_1)=y_0$, then $H(x_1,y_0) = 0$ and Rolle's theorem implies that
\begin{itemize}
\item If $x_1 < x_0$, then the function $x \mapsto \frac{\partial H}{\partial x}(x,y_0)$ has at least one zero in each of the open intervals $(0,x_1)$, $(x_1,x_0)$, and $(x_0,+\infty)$;
\item If $x_0<x_1$, then the function $x \mapsto \frac{\partial H}{\partial x}(x,y_0)$ has at least one zero in each of the open intervals $(0,x_0)$, $(x_0,x_1)$, and $(x_1,+\infty)$.
\end{itemize}
Since $x_0$ is also a zero of the function $x \mapsto \frac{\partial H}{\partial x}(x,y_0)$ by \eqref{681}, we deduce that the polynomial $Q(t)$ has at least four distinct zeros in $(0,+\infty)$, a contradiction. 
In conclusion, $x_0$ is the unique solution of the equation $Y_0(x)=y_0$ in the interval $(\eta_0,\eta_1)$.
\end{proof}
		
\begin{lemma}
\label{y'+-}
There exists a unique $x_0 \in (\eta_0,\eta_1)$ such that
\[
Y_0^{\prime}(x)
\begin{cases}
>0 &\text{if}\; x \in (\eta_0,x_0),\\
=0 &\text{if}\; x = x_0,\\
<0 &\text{if}\; x \in (x_0,\eta_1).
\end{cases}
\]
\end{lemma}
		
\begin{proof}
By Equation \eqref{Y_0_is_continuous_at_end_points} and the fact $Y_0 \in C^{1}((\eta_0,\eta_1),\mathbb{R}_{>0})$, there exists a maximum point $x_0 \in (\eta_0,\eta_1)$ of the function $Y_0$; we have $Y_0^{\prime}(x_0)=0$. Lemma \ref{y'(x)=0} implies that $x_0$ is the unique maximum point of the function $Y_0$. 
For any $\hat{x} \in (\eta_0,\eta_1) \setminus \{ x_0 \}$, we have $0<Y_0(\hat{x})<Y_0(x_0)$. 
By the intermediate value theorem, the equation $Y_0(x)=Y_0(\hat{x})$ has at least two solutions: one in $(\eta_0,x_0)$ and another in $(x_0,\eta_1)$. 
Therefore, Lemma \ref{y'(x)=0} implies that $Y_0^{\prime}(\hat{x}) \neq 0$ for any $\hat{x} \in (\eta_0,\eta_1) \setminus \{ x_0 \}$. 
Thus, the continuous function $Y_0^{\prime}$ does not change sign on each of the intervals $(\eta_0,x_0)$ and $(x_0,\eta_1)$. 
Finally, Lagrange's mean value theorem implies that $Y_0^{\prime}(\xi_0) > 0$ and $Y_0^{\prime}(\xi_1)<0$ for some $\xi_0 \in (\eta_0,x_0)$ and $\xi_1 \in (x_0,\eta_1)$. 
We conclude that $Y_0^{\prime}(x) > 0$ for any $x \in (\eta_0,x_0)$ and $Y_0^{\prime}(x)<0$ for any $x \in (x_0,\eta_1)$.
\end{proof}
		
\section{Asymptotic estimates}
\label{Sect_asym_S_n}
		
In this section, our goal is to investigate the asymptotic behavior of $S_n$ as $n \to +\infty$. 
Throughout this section, we assume that $k,q,r,n$ are positive integers such that $k \geqslant 2$, $q\geqslant 3$, $r>2k$, and $n \in q!\mathbb{N}$. 
The functions $f(z)$ and $g(z)$ are defined in Definition \ref{def_f_g}. 
The function $g_n(z)$ is defined by \eqref{def_g_n}.
		
Recall that the substitution $z = r(w-1)/2$ converts the function $f'(z)$ to the function $h(w)$ (see \eqref{why_h}):
\begin{align}
&~f'(z) = f'\left(\frac{r(w-1)}{2}\right) \notag\\
=&~(a+b)\big(\log(w-1)-\log(w+1)\big)+b\big(\log(1+s+w)-\log(1+s-w)\big) \notag\\
=&~h(w),\quad\text{where}\ a=q,\ b=k,\ \text{and}\ s=2q/r. \label{f'_to_h}
\end{align}
Applying the substitution \eqref{f'_to_h} to Lemmas \ref{H(x,0)}, \ref{H(x,y)}, \ref{im(h(z))}, and \ref{y'+-}, we obtain the following Lemma \ref{Y}.
		
\begin{lemma}
\label{Y}
There exist a unique $\mu_0 \in (0,q)$, a unique $\mu_1 \in (q,+\infty)$, and a $C^{1}$ function $Y\colon (\mu_0,\mu_1) \rightarrow \mathbb{R}_{>0}$ such that
\begin{equation}
\label{sign_of_Re_f'}
\operatorname{Re}(f'(x+iy))
\begin{cases}
<0 &\text{if}\; x \in (-r/2,\mu_0] \cup [\mu_1,+\infty) \;\text{and}\; y > 0, \\
<0 &\text{if}\; x \in (\mu_0,\mu_1) \;\text{and}\; y > Y(x), \\
=0 &\text{if}\; x \in (\mu_0,\mu_1) \;\text{and}\; y = Y(x), \\
>0 &\text{if}\; x \in (\mu_0,\mu_1) \;\text{and}\; 0< y < Y(x). \\
\end{cases}
\end{equation}
We have
\begin{equation}
\label{Y_is_continuous_at_end_points}
\lim_{x \to \mu_0^+} Y(x)=0, \quad \lim_{x \to \mu_1^-} Y(x)=0, 
\end{equation}
and
\begin{equation}
\label{Re(f')_partial_y<0_on_C}
\frac{\partial u}{\partial y}(x,Y(x)) < 0 \quad \text{for any~} x\in (\mu_0,\mu_1),\text{~where~} u(x,y)=\operatorname{Re}(f'(x+iy)).
\end{equation}
There exists a unique $x_* \in (\mu_0,\mu_1)$ such that
\begin{equation}
\label{sign_of_Y'}
Y^{\prime}(x)
\begin{cases}
>0 &\text{if}\; x \in (\mu_0,x_*),\\
=0 &\text{if}\; x = x_*,\\
<0 &\text{if}\; x \in (x_*,\mu_1).
\end{cases}
\end{equation}
Moreover, we have $\lim\limits_{x \to \mu_0^+} \operatorname{Im}(f'(x+iY(x)))=0$, $\lim\limits_{x \to \mu_1^-} \operatorname{Im}(f'(x+iY(x)))=k\pi$, and
\begin{equation}
\label{Im(f')_increase_on_C}
x \mapsto \operatorname{Im}(f'(x+iY(x))) \quad\text{increases strictly on~} (\mu_0,\mu_1).
\end{equation}
\end{lemma}
		
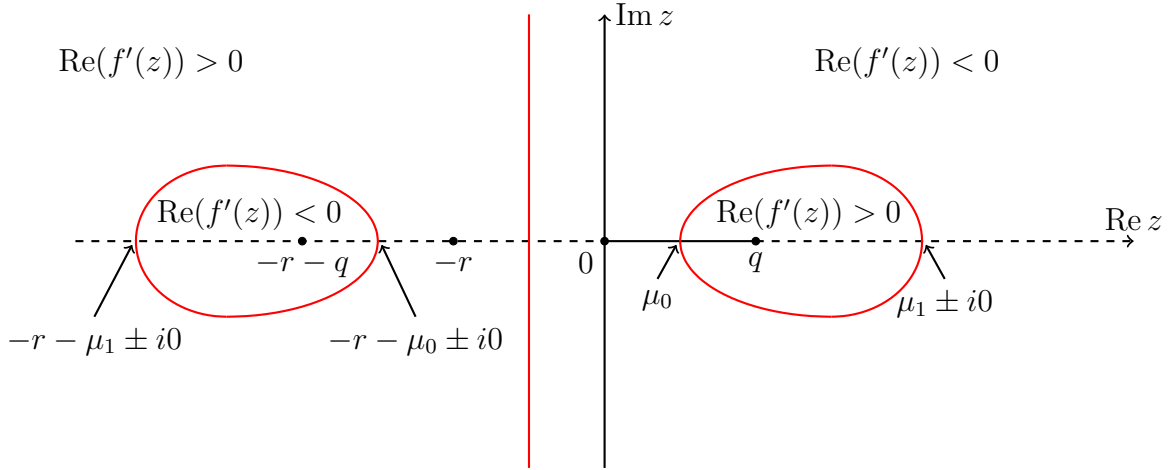
\begin{figure}[h]
\centering
\begin{tikzpicture}
\draw[thick] [->] (1,-3)--(1,3);
\draw[thick]  (1,0)--(3,0);
\draw [dashed, thick] (-6,0)--(1,0);
\draw [dashed, thick][->] (3,0)--(8,0);
\draw [red, thick] (0,-3)--(0,3);
\draw[red, domain=90:270, smooth, variable=\t,thick]
plot ({2*cos(\t)+4}, {1*sin(\t)});
\draw[red, domain=-90:90, smooth, variable=\t,thick]
plot ({1.2*cos(\t)+4}, {1*sin(\t)});
\draw[red, domain=-90:90, smooth, variable=\t,thick]
plot ({2*cos(\t)-4}, {1*sin(\t)});
\draw[red, domain=90:270, smooth, variable=\t,thick]
plot ({1.2*cos(\t)-4}, {1*sin(\t)});
\draw[fill] (1,0) circle [radius=0.05];
\draw[fill] (-1,0) circle [radius=0.05];
\draw[fill] (3,0) circle [radius=0.05];
\draw[fill] (-3,0) circle [radius=0.05];
\node [below left] at (1,0) {$0$};
\node [below] at (-1,0) {$-r$};
\node [below] at (3,0) {$q$};
\node [below] at (-3,0) {$-r-q$};
\draw [thick] [->] (1.7,-0.5)--(1.95,-0.05);
\node [below] at (1.7,-0.5) {$\mu_0$};
\draw [thick] [->] (-1.5,-1)--(-1.95,-0.05);
\node [below] at (-1.5,-1) {$-r-\mu_0 \pm i0$};
\draw [thick] [->] (-5.75,-1)--(-5.25,-0.05);
\node [below] at (-5.75,-1) {$-r-\mu_1 \pm i 0$};
\draw [thick] [->] (5.5,-0.5)--(5.25,-0.05);
\node [below] at (5.5,-0.5) {$\mu_1 \pm i0$};
\node [right] at (1,3) {$\operatorname{Im}z$};
\node [above] at (8,0) {$\operatorname{Re}z$};
\node [above] at (3.7,0) {$\operatorname{Re}(f'(z))>0$};
\node [above] at (5,2) {$\operatorname{Re}(f'(z))<0$};
\node [above] at (-3.7,0) {$\operatorname{Re}(f'(z))<0$};
\node [above] at (-5,2) {$\operatorname{Re}(f'(z))>0$};
\end{tikzpicture}
\caption{solutions of $\operatorname{Re}(f'(z))=0$ (red).}
\label{fig:sol_of_Re(f'(z))=0}
\end{figure}
		
\begin{lemma}
\label{roots_of_f'}
Consider solutions of the equation 
\begin{equation}
\label{f'=lambda}
f'(z)=\lambda \pi i
\end{equation}
in the domain $\mathbb{C} \setminus \big( (-\infty,0] \cup [q,+\infty) \big)$.
\begin{enumerate}
\item[\textup{(1)}] For $\lambda \in (0,k)$, Equation \eqref{f'=lambda} has exactly two solutions $\tau_{\lambda}$ and $-r-\overline{\tau_{\lambda}}$, where $\tau_{\lambda} = x_{\lambda} + iY(x_{\lambda})$ for some $x_\lambda \in (\mu_0,\mu_1)$.
For $\lambda = 0$, Equation \eqref{f'=lambda} has a unique solution $\tau_0 := \mu_0$. 
				
\item[\textup{(2)}] The polynomial
\begin{equation}
\label{def_P}
P(z) := (z+r)^{q+k}(z-q)^k - z^{q+k}(z+q+r)^k
\end{equation}
has no multiple zero. 
It has exactly $k$ zeros in the half-plane $\operatorname{Re} z > -r/2$, including $\tau_{k-2}$.
\end{enumerate}
\end{lemma}
		
\begin{proof}
Part (1) immediately follows from Substitution \eqref{f'_to_h} and Lemma \ref{roots_of_h}. 
Now we prove part (2). 
Note that if $z \in \mathbb{C} \setminus \mathbb{R}$, then $P(z) = 0$ if and only if $f'(z) = \lambda \pi i$ for some integer $\lambda$ such that $\lambda \equiv k \pmod{2}$. 
Therefore, Substitution \eqref{f'_to_h} and Lemma \ref{roots_of_h} imply that $P(z)$ has
\begin{itemize}
\item $q-1$ distinct zeros on the line $\operatorname{Re} z = -r/2$; 
\item $k$ distinct zeros in the half-plane $\operatorname{Re} z > -r/2$, including $\tau_{k-2}$;
\item $k$ distinct zeros in the half-plane $\operatorname{Re} z < -r/2$.
\end{itemize}
Since $\deg P = q+2k-1$, we conclude that part (2) holds.
\end{proof}
		
\bigskip
		
Now, we investigate the asymptotic behavior of the integral $J_{n,\lambda}$ (defined by \eqref{def_J}) as $n \to +\infty$.
		
\begin{lemma}
\label{f_0(z)}
For any $\lambda\in[0,k)$, the asymptotic behavior of the integral \eqref{def_J} as $n\to+\infty$ is determined by the single saddle point $\tau_{\lambda}$ defined in Lemma \ref{roots_of_f'}. 
More precisely, the following asymptotic formula holds:
\begin{align*}
J_{n,\lambda}&=\frac{1}{2\pi i}\int_{\mu-i\infty}^{\mu+i\infty}e^{n(f(z)-\lambda\pi iz)}g_n(z)\,\de z\\
&\sim \frac{1}{\sqrt{2\pi n|f''(\tau_\lambda)|}}e^{nf_0(\tau_\lambda)}|g(\tau_\lambda)|\cdot e^{-\frac{i}{2}\arg(f''(\tau_\lambda))+i\arg(g(\tau_\lambda))} \quad \text{as~} n\to+\infty,
\end{align*}
where
\begin{align}
f_0(z) :=& f(z)-f'(z)z \notag\\
=&~k(r+q)\log(z+r+q)+kq\log(q-z)-r(q+k)\log(z+r) \notag\\
&+ rq\log(r)+2kq\sum_{p \mid q \atop p \text{~prime}} \frac{\log p}{p-1}. 
\label{def_f_0}
\end{align}
\end{lemma}
		
\begin{proof}
Our strategy is to choose a contour $\mathcal{L}$ that passes through the saddle point $\tau_{\lambda}$ and satisfies all requirements of Theorem \ref{saddle} (with $f(z)$ in Theorem \ref{saddle} replaced by $f(z)-\lambda \pi i z$). 
			
Define the curve
\[
C := \left\{ x+iY(x) \mid x \in (\mu_0,\mu_1) \right\}.
\]
Note that $\tau_{\lambda} \in \{\mu_0\} \cup C$ for any $\lambda \in [0,k)$. 
Clearly, we have $g(\tau_{\lambda}) \neq 0$ by \eqref{def_g}. 
By \eqref{def_f}, we have
\begin{equation}
\label{f''}
\frac{\de^2}{\de z^2} (f(z)-\lambda \pi i z) = f''(z) = \frac{(r-2k)q(2z+r)^2 - rq(r+2q)(r+2k+2q)}{4z(z+r)(z+r+q)(z-q)}.
\end{equation}
The function $f''(z)$ has only real zeros and $f''(\mu_0) >0$ (since $0<\mu_0<q$). 
Thus, we have $f''(\tau_\lambda) \neq 0$ for any $\lambda \in [0,k)$. 
Therefore, requirement (2) of Theorem \ref{saddle} is satisfied.
The requirement (1) of Theorem \ref{saddle} is easy to check for each case below.
			
In the sequel, we will consider four cases. 
We mention in advance that for each case, we can always choose a small $\varepsilon_0 \in (0,1)$ such that the contour $\mathcal{L}$ lies within the domain $D_{\varepsilon_0}$ defined by \eqref{def_D_eps}. The rectifiability condition \eqref{rectifiability} can be verified directly for $\mathcal{L}$.
Requirement (5) of Theorem \ref{saddle} is satisfied by \eqref{g_n=g}.
Requirement (6) of Theorem \ref{saddle} is satisfied by Lemma \ref{lem_general_L} (2) or by the proof of Lemma \ref{lem_S=sum_J}. 
Thus, we only need to check requirements (3) and (4) of Theorem \ref{saddle} for each case.
			
Let $z=z(t)$ be the parametrization of $\mathcal{L}$ and let $z(t_0)=\tau_{\lambda}$.
Recall that the requirement (3) of Theorem \ref{saddle} is
\begin{equation}
\label{requirement_(4)}
\operatorname{Re}\big(f''(z(t_0)) \cdot z'(t_0)^2\big) < 0.
\end{equation}
Thus, for each case, we only need to verify that $z=\tau_\lambda$ is the unique maximum point of $\operatorname{Re}(f(z)-\lambda \pi i z)$ along $\mathcal{L}$ and Equation \eqref{requirement_(4)} holds. 
			
Let $x_{\lambda}:=\operatorname{Re}(\tau_{\lambda})$. Now, we distinguish between four cases.
			
\bigskip
			
\textbf{Case 1:} $x_{\lambda}=\mu_0$ (that is, $\lambda = 0$ and $\tau_{\lambda}=\mu_0$). 
We choose $\mathcal{L}$ to be the upward vertical line $\operatorname{Re} z = \mu_0$, as shown in Figure \ref{case1}.
\begin{figure}[h]
\centering
\begin{tikzpicture}
\draw [thick] [->] (-1,0)--(7,0);
\draw [thick] [->] (0,-3)--(0,4);
\draw[red, domain=90:180, smooth, variable=\t,thick] plot ({2*cos(\t)+4}, {1*sin(\t)});
\draw[red, domain=0:90, smooth, variable=\t,thick] plot ({1.2*cos(\t)+4}, {1*sin(\t)});
\draw[fill] (0,0) circle [radius=0.05];
\draw[fill] (5.2,0) circle [radius=0.05];
\draw[fill] (2,0) circle [radius=0.05];
\draw [thick,blue] [->] (2,-3)--(2,-1);
\draw [thick,blue]  (2,-1)--(2,0);
\draw [thick,blue] [->] (2,0)--(2,2);
\draw [thick,blue]  (2,2)--(2,4);
\node [below left] at (0,0) {$0$};
\node [red,below left] at (2,0) {$\mu_0$};
\node [red,below] at (5.2,0) {$\mu_1$};
\node [above] at (7,0) {$\operatorname{Re}z$};
\node [right] at (0,4) {$\operatorname{Im}z$};
\node [red,above] at (3.7,1) {$C$};
\end{tikzpicture}
\caption{contour $\mathcal{L}$ for Case 1 (blue).}
\label{case1}
\end{figure}
			
We parametrize $\mathcal{L}$ by $z(t) = \mu_0+it$, $t \in (-\infty,+\infty)$. Then 
\begin{equation}
\label{999}
\frac{\de}{\de t} \operatorname{Re}(f(z(t))) = \operatorname{Re}\left( f'(z(t)) \cdot z'(t) \right) = -\operatorname{Im}(f'(z(t))).
\end{equation}
By \eqref{def_f}, we have
\[
\operatorname{Im}(f'(z)) = (q+k)(\arg z - \arg(z+r)) + k(\arg(z+r+q)-\arg(q-z)),
\]
where each $\arg(\cdot)$ takes values in $(-\pi,\pi)$. 
It is easy to check that
\begin{equation}
\label{sign_of_Im(f')}
\operatorname{sgn}(\operatorname{Im}(f'(z))) = \operatorname{sgn}(\operatorname{Im} z). 
\end{equation}
By Equations \eqref{999} and \eqref{sign_of_Im(f')}, we have
\[
\frac{\de}{\de t} \operatorname{Re}(f(z(t))) 
\begin{cases}
>0 &\text{if}\; t<0,\\
<0 &\text{if}\; t>0.
\end{cases}
\]
Thus, $z(0)=\mu_0$ is the unique maximum point of $\operatorname{Re}(f(z))$ along $\mathcal{L}$. 
Moreover, by \eqref{f''} and $0<\mu_0<q$, we have
\[
\operatorname{Re}(f''(z(0)) \cdot z'(0)^2) = -f''(\mu_0) < 0.
\]
Therefore, Equation \eqref{requirement_(4)} is satisfied.
			
\bigskip
			
\textbf{Case 2:} $x_\lambda \in (\mu_0,x_*)$. We choose $\mathcal{L}$ to be the contour parametrized by
\[
z(t)=
\begin{cases}
\mu_0+i(t-\mu_0) &\text{if}\; t\in(-\infty,\mu_0],\\
t+iY(t) &\text{if}\; t\in(\mu_0,x_*],\\
t+iY(x_*) &\text{if}\; t\in(x_*,+\infty).
\end{cases}
\]
See Figure \ref{case2}.			
			
\begin{figure}[h]
\centering
\begin{tikzpicture}
\draw [thick] [->] (-1,0)--(7,0);
\draw [thick] [->] (0,-3)--(0,4);
\draw[blue, domain=90:180, smooth, variable=\t,thick] plot ({2*cos(\t)+4}, {1*sin(\t)});
\draw[red, domain=0:90, smooth, variable=\t,thick] plot ({1.2*cos(\t)+4}, {1*sin(\t)});
\draw[fill] (0,0) circle [radius=0.05];
\draw[fill] (4,0) circle [radius=0.05];
\draw[fill] (2,0) circle [radius=0.05];
\draw[fill] (4,1) circle [radius=0.05];
\draw[fill] (5.2,0) circle [radius=0.05];
\draw[fill] (2.5,0) circle [radius=0.05];
\draw[fill] (2.5,0.66) circle [radius=0.05];
\draw [thick,blue] [->] (2,-3)--(2,-1);
\draw [thick,blue]  (2,-1)--(2,0);
\draw [thick,blue] [->] (4,1)--(6,1);
\draw [thick,blue]  (6,1)--(7,1);
\draw [thick,dashed]  (4,1)--(4,0);
\draw [thick,dashed]  (2.5,0)--(2.5,0.66);
\node [below] at (4,0) {$x_*$};
\node [below] at (2.5,0) {$x_{\lambda}$};
\node [above] at (2.5,0.7) {$\tau_{\lambda}$};
\node [below left] at (0,0) {$0$};
\node [red,below left] at (2,0) {$\mu_0$};
\node [red,below] at (5.2,0) {$\mu_1$};
\node [above] at (7,0) {$\operatorname{Re}z$};
\node [right] at (0,4) {$\operatorname{Im}z$};
\node [red,above] at (3.7,1) {$C$};
\end{tikzpicture}
\caption{contour $\mathcal{L}$ for Case 2 (blue).}
\label{case2}
\end{figure}
			
On the half-line $z(t)=\mu_0+i(t-\mu_0)$, $t\in(-\infty,\mu_0]$, we have
\begin{align*}
&~\frac{\de }{\de t}\operatorname{Re}\Big(f(z(t))-\lambda\pi iz(t)\Big)=\operatorname{Re}\Big((f'(z(t))-\lambda\pi i) \cdot z'(t)\Big)\\
=&~-\operatorname{Im}\big(f'(z(t))\big)+\lambda\pi\geqslant\lambda\pi>0\qquad(\text{by}\ \eqref{sign_of_Im(f')}).
\end{align*}
			
On the curve $z(t)=t+iY(t)$, $t\in(\mu_0,x_*)$, we have
\begin{align*}
&~\frac{\de }{\de t}\operatorname{Re}\Big(f(z(t))-\lambda\pi iz(t)\Big)=\operatorname{Re}\Big((f'(z(t))-\lambda\pi i) \cdot z'(t)\Big)\\
=&~\operatorname{Re}\Big((f'(z(t))-\lambda\pi i) \cdot(1+iY'(t))\Big) = -Y'(t)\cdot \Big(\operatorname{Im}\big(f'(z(t))\big)-\lambda\pi\Big).
\end{align*}
By \eqref{sign_of_Y'}, we have $Y'(t)>0$ for $t \in (\mu_0,x_*)$. 
By \eqref{Im(f')_increase_on_C}, we have $\operatorname{Im}(f'(z(t))) < \lambda \pi$ for $t \in (\mu_0,x_{\lambda})$ and $\operatorname{Im}(f'(z(t))) > \lambda \pi$ for $t \in (x_{\lambda},x_{*})$. Hence,
\[
\frac{\de }{\de t}\operatorname{Re}\Big(f(z(t))-\lambda\pi iz(t)\Big)
\begin{cases}
>0 &\text{if}\; t\in(\mu_0,x_{\lambda}),\\
<0 &\text{if}\; t\in(x_{\lambda},x_{*}).
\end{cases}
\]
			
On the half-line $z(t)=t+iY(x_*)$, $t\in(x_*,+\infty)$, we have
\begin{align*}
&~\frac{\de }{\de t}\operatorname{Re}\Big(f(z(t))-\lambda\pi iz(t)\Big)=\operatorname{Re}\Big((f'(z(t))-\lambda\pi i) \cdot z'(t)\Big)\\
=&~\operatorname{Re}\Big(f'(z(t))\Big) <0\qquad(\text{by}\ \eqref{sign_of_Re_f'}).
\end{align*}
			
In summary, we have shown that $z(x_{\lambda})=\tau_{\lambda}$ is the unique maximum point of $\operatorname{Re}(f(z)-\lambda\pi i z)$ along $\mathcal{L}$. 
			
Write $f'(z)=u(x,y)+iv(x,y)$. 
By the Cauchy--Riemann equation, we have
\begin{equation}
\label{CR}
\frac{\partial v}{\partial x} = -\frac{\partial u}{\partial y}.
\end{equation}
By \eqref{sign_of_Re_f'}, we have $u(x,Y(x)) = 0$ for any $x \in (\mu_0,\mu_1)$, which implies
\begin{equation}
\label{u_x_to_u_y_on_C}
\frac{\partial u}{\partial x}(x,Y(x)) =  - \frac{\partial u}{\partial y}(x,Y(x)) \cdot Y'(x) \quad\text{for any}\ x \in (\mu_0,\mu_1).
\end{equation}
For $t \in (\mu_0,x_{*})$, we have $z(t)=t+iY(t)$ and
\begin{align}
&~\operatorname{Re}\Big( f''(z(t)) \cdot z'(t)^2 \Big) = \operatorname{Re}\left( \left( \frac{\partial u}{\partial x}(t,Y(t)) + i\frac{\partial v}{\partial x}(t,Y(t)) \right) \cdot \big( 1+iY'(t) \big)^2 \right) \notag\\
=&~\frac{\partial u}{\partial x}(t,Y(t))\cdot\big(1-Y'(t)^2\big)-\frac{\partial v}{\partial x}(t,Y(t))\cdot 2Y'(t).\label{777}
\end{align}
Substituting \eqref{CR} and \eqref{u_x_to_u_y_on_C} into \eqref{777}, we obtain
\begin{equation}
\label{778}
\operatorname{Re}\Big( f''(z(t)) \cdot z'(t)^2 \Big) = \frac{\partial u}{\partial y}(t,Y(t)) \cdot Y'(t) \cdot \big(1+Y'(t)^2\big).
\end{equation}
By Equations \eqref{778}, \eqref{sign_of_Y'}, and \eqref{Re(f')_partial_y<0_on_C}, we obtain
\[
\operatorname{Re}\Big( f''(z(t)) \cdot z'(t)^2 \Big)<0\quad\text{for any}\ t\in(\mu_0,x_{*}).
\]
In particular, 
\[
\operatorname{Re}\Big( f''(z(x_\lambda)) \cdot z'(x_\lambda)^2 \Big)<0.
\]
Therefore, Equation \eqref{requirement_(4)} is satisfied.
			
\bigskip
			
\textbf{Case 3:} $x_{\lambda}=x_*$. 
Take a real number $\hat{x}<x_*$ that is sufficiently close to $x_*$ such that $\delta := Y(x_*) - Y(\hat{x}) >0$ is sufficiently small. 
We choose $\mathcal{L}$ to be the contour parametrized by
\[
z(t) =
\begin{cases}
\mu_0+i(t-\mu_0) &\text{if}\; t\in(-\infty,\mu_0],\\
t+iY(t) &\text{if}\; t\in(\mu_0,\hat{x}],\\
t+i(Y(x_*)-\delta) &\text{if}\; t\in(\hat{x},x_*-\delta],\\
t+i(t-x_*+Y(x_*)) &\text{if}\; t\in(x_*-\delta,x_*+\delta],\\
t+i(Y(x_*)+\delta) &\text{if}\; t\in(x_*+\delta,+\infty).
\end{cases}
\]
See Figure \ref{case3}.
			
\begin{figure}[h]
\centering
\begin{tikzpicture}
\draw [thick] [->] (-1,0)--(7,0);
\draw [thick] [->] (0,-3)--(0,4);
\draw[red, domain=90:125, smooth, variable=\t,thick] plot ({2*cos(\t)+4}, {1*sin(\t)});
\draw[blue, domain=125:180, smooth, variable=\t,thick] plot ({2*cos(\t)+4}, {1*sin(\t)});
\draw[red, domain=0:90, smooth, variable=\t,thick] plot ({1.2*cos(\t)+4}, {1*sin(\t)});
\draw[fill] (0,0) circle [radius=0.05];
\draw[fill] (4,0) circle [radius=0.05];
\draw[fill] (2.82,0.8) circle [radius=0.05];
\draw[fill] (2.82,0) circle [radius=0.05];
\draw[fill] (4,1) circle [radius=0.05];
\draw[fill] (5.2,0) circle [radius=0.05];
\draw[fill] (2,0) circle [radius=0.05];
\draw [thick,blue] [->] (2,-3)--(2,-1);
\draw [thick,blue]  (2,-1)--(2,0);
\draw [thick,blue] (2.82,0.8)--(3.8,0.8);
\draw [thick,blue] (3.8,0.8)--(4.2,1.2);
\draw [thick,blue] [->] (4.2,1.2)--(5.5,1.2);
\draw [thick,blue]  (5.5,1.2)--(7,1.2);
\draw [thick,dashed]  (4,1)--(4,0);
\draw [thick,dashed]  (2.82,0.8)--(2.82,0);
\node [below] at (4,0) {$x_*$};
\node [below] at (2.82,0) {$\hat{x}$};
\node [below left] at (0,0) {$0$};
\node [red,below left] at (2,0) {$\mu_0$};
\node [red,below] at (5.4,0) {$\mu_1$};
\node [above] at (7,0) {$\operatorname{Re}z$};
\node [right] at (0,4) {$\operatorname{Im}z$};
\node [red,above] at (3.4,1) {$C$};
\end{tikzpicture}
\caption{contour $\mathcal{L}$ for Case 3 (blue).}
\label{case3}
\end{figure}
			
For $t \in [x_*-\delta, x_*+\delta]$, we have $z(t)=t+i(t-x_*+Y(x_*))$ and
\[
f(z(t))-\lambda\pi i z(t) = f(z(x_*))-\lambda\pi i z(x_*) + \frac{f''(z(x_*))\cdot z'(x_*)^2}{2}(t-x_*)^2 + O(|t-x_*|^3).
\]
Write $f'(z)=u(x,y)+iv(x,y)$. 
Then
\begin{align}
\operatorname{Re}\left(f''(z(x_*))\cdot(z'(x_*))^2\right)&=\operatorname{Re}\left(\left(\frac{\partial u}{\partial x}(x_*,Y(x_*))+i\frac{\partial v}{\partial x}(x_*,Y(x_*))\right)\cdot(1+i)^2\right)\notag\\
&=-2\frac{\partial v}{\partial x}(x_*,Y(x_*))\notag\\
&= 2\frac{\partial u}{\partial y}(x_*,Y(x_*)) < 0, \label{790}
\end{align} 		
where in the last line, we have used the Cauchy--Riemann equation and \eqref{Re(f')_partial_y<0_on_C}.
Therefore, 
\begin{align*}
&~\operatorname{Re}\left( f(z(t))-\lambda\pi iz(t) \right)\\
=&~\operatorname{Re}\left( f(z(x_*))-\lambda\pi iz(x_*)\right) + \underbrace{\frac{\partial u}{\partial y}(x_*,Y(x_*))}_{<0}(t-x_*)^2+O(|t-x_*|^3) 
\end{align*}
for any $t\in [x_*-\delta, x_*+\delta]$.
We deduce that $z(x_*)=\tau_{\lambda}$ is the unique maximum point of $\operatorname{Re}(f(z)-\lambda\pi i z)$ on the segment $z(t)=t+i(t-x_*+Y(x_*))$, $t\in [x_*-\delta, x_*+\delta]$. 
			
Using arguments similar to those in \textbf{Case 2}, one can show that $t \mapsto \operatorname{Re}(f(z(t))-\lambda\pi i z(t))$ increases strictly on $(-\infty,x_*-\delta]$ and decreases strictly on $[x_*+\delta,+\infty)$. 
Thus, $z(x_*)=\tau_{\lambda}$ is the unique maximum point of $\operatorname{Re}(f(z)-\lambda\pi i z)$ along $\mathcal{L}$.
Moreover, Equation \eqref{790} implies that Equation \eqref{requirement_(4)} is satisfied.
			
\bigskip
			
\textbf{Case 4:} $x_\lambda\in(x_*,\mu_1)$. 
By \eqref{sign_of_Y'} and the intermediate value theorem, there exists a unique $\hat{x} \in (\mu_0,x_*)$ such that $Y(\hat{x})=Y(x_{\lambda})$. 
We choose $\mathcal{L}$ to be the contour parametrized by
\[
z(t)=
\begin{cases}
\mu_0+i(t-\mu_0) &\text{if}\; t\in(-\infty,\mu_0],\\
t+iY(t) &\text{if}\; t\in(\mu_0,\hat{x}],\\
t+iY(x_\lambda) &\text{if}\; t\in(\hat{x},+\infty).
\end{cases}
\]
See Figure \ref{case4}.
			
\begin{figure}[h]
\centering
\begin{tikzpicture}
\draw [thick] [->] (-1,0)--(7,0);
\draw [thick] [->] (0,-3)--(0,4);
\draw[red, domain=90:125, smooth, variable=\t,thick] plot ({2*cos(\t)+4}, {1*sin(\t)});
\draw[blue, domain=125:180, smooth, variable=\t,thick] plot ({2*cos(\t)+4}, {1*sin(\t)});
\draw[red, domain=0:90, smooth, variable=\t,thick] plot ({1.2*cos(\t)+4}, {1*sin(\t)});
\draw[fill] (0,0) circle [radius=0.05];
\draw[fill] (4,0) circle [radius=0.05];
\draw[fill] (2.82,0.8) circle [radius=0.05];
\draw[fill] (4.7,0) circle [radius=0.05];
\draw[fill] (4.7,0.8) circle [radius=0.05];
\draw[fill] (2.82,0) circle [radius=0.05];
\draw[fill] (4,1) circle [radius=0.05];
\draw[fill] (5.2,0) circle [radius=0.05];
\draw[fill] (2,0) circle [radius=0.05];
\draw [thick,blue] [->] (2,-3)--(2,-1);
\draw [thick,blue]  (2,-1)--(2,0);
\draw [thick,blue] [->] (2.82,0.8)--(6,0.8);
\draw [thick,blue]  (6,0.8)--(7,0.8);
\draw [thick,dashed]  (4,1)--(4,0);
\draw [thick,dashed]  (2.82,0.8)--(2.82,0);
\draw [thick,dashed]  (4.7,0)--(4.7,0.8);
\node [below] at (4,0) {$x_*$};
\node [below] at (4.7,0) {$x_\lambda$};
\node [above] at (4.7,0.9) {$\tau_{\lambda}$};
\node [below] at (2.7,0) {$\hat{x}$};
\node [below left] at (0,0) {$0$};
\node [red,below left] at (2,0) {$\mu_0$};
\node [red,below] at (5.4,0) {$\mu_1$};
\node [above] at (7,0) {$\operatorname{Re}z$};
\node [right] at (0,4) {$\operatorname{Im}z$};
\node [red,above] at (3.7,1) {$C$};
\end{tikzpicture}
\caption{contour $\mathcal{L}$ for Case 4 (blue).}
\label{case4}
\end{figure}
			
Using arguments similar to those in \textbf{Case 2}, one can show that $t \mapsto \operatorname{Re}(f(z(t))-\lambda\pi i z(t))$ increases strictly on $(-\infty,x_{\lambda})$ and decreases strictly on $(x_{\lambda},+\infty)$. 
Thus, $z(x_{\lambda})=\tau_{\lambda}$ is the unique maximum point of $\operatorname{Re}(f(z)-\lambda\pi i z)$ along $\mathcal{L}$.
Write $f'(z)=u(x,y)+iv(x,y)$. 
We have
\begin{align}
\operatorname{Re}\left( f''(z(x_{\lambda})) \cdot z'(x_{\lambda})^2\right) &= \frac{\partial u}{\partial x}(x_{\lambda},Y(x_{\lambda}))\notag\\
&= -\frac{\partial u}{\partial y}(x_\lambda,Y(x_\lambda)) \cdot Y'(x_\lambda) \quad(\text{by}\ \eqref{u_x_to_u_y_on_C}).\label{799}
\end{align}
By \eqref{799}, \eqref{Re(f')_partial_y<0_on_C}, and \eqref{sign_of_Y'}, we obtain
\[
\operatorname{Re}\left( f''(z(x_{\lambda})) \cdot z'(x_{\lambda})^2\right) < 0.
\]
Therefore, Equation \eqref{requirement_(4)} is satisfied.
			
\bigskip
			
In conclusion, for any $\lambda \in [0,k)$, we obtain from Theorem \ref{saddle} that 
\[
J_{n,\lambda} \sim \frac{1}{\sqrt{2\pi n|f''(\tau_\lambda)|}}e^{n(f(\tau_\lambda)-\lambda \pi i \tau_{\lambda})}|g(\tau_\lambda)|\cdot e^{-\frac{i}{2}\arg(f''(\tau_\lambda))+i\arg(g(\tau_\lambda))} \quad \text{as~} n\to+\infty.
\]
Finally, noting that $f'(\tau_\lambda)=\lambda \pi i$, we have $f(\tau_\lambda)-\lambda \pi i \tau_{\lambda} = f_0(\tau_{\lambda})$ by the definition of $f_0$ (see \eqref{def_f_0}). 
The proof of Lemma \ref{f_0(z)} is complete.
\end{proof}
		
\begin{lemma}
\label{Re(f_0)}
Let $\tau_{\lambda}$ be as defined in Lemma \ref{roots_of_f'}.
Let $f_0(z)$ be the function in \eqref{def_f_0}.
Then, the function $\lambda \mapsto \operatorname{Re}(f_0(\tau_{\lambda}))$ is strictly increasing on $[0,k)$.
\end{lemma}
		
\begin{proof}
Write $z=x+iy$ and $f'(z)=u(x,y)+iv(x,y)$. 
Recall that for any $\lambda \in (0,k)$ we have $\tau_{\lambda}=x_{\lambda}+iY(x_{\lambda})$ for some $x_{\lambda} \in (\mu_0,\mu_1)$. 
By $f'(\tau_{\lambda})=\lambda \pi i$ and Equations \eqref{Im(f')_increase_on_C} and \eqref{Y_is_continuous_at_end_points}, 
\begin{equation}
\lambda \mapsto x_{\lambda} \quad\text{is strictly increasing on}\ [0,k).
\label{741}
\end{equation}
Let $z(x) = x+iY(x)$ for $x \in (\mu_0,\mu_1)$.
We have
\begin{align*}
\frac{\de}{\de x} \operatorname{Re}\left( f_0(z(x)) \right) &=\frac{\de}{\de x} \operatorname{Re}\left( f(z(x)) - z(x)f'(z(x)) \right) = \operatorname{Re}\left( -z(x)f''(z(x)) \cdot z'(x) \right) \\
&=\operatorname{Re}\left( -(x+iY(x)) \cdot \left( \frac{\partial u}{\partial x}(x,Y(x)) +i\frac{\partial v}{\partial x}(x,Y(x)) \right) \cdot (1+iY'(x)) \right)\\
&= -x \cdot \frac{\partial u}{\partial x} + Y(x) \cdot \frac{\partial v}{\partial x} + Y'(x) \cdot \left( Y(x) \cdot \frac{\partial u}{\partial x} + x \cdot \frac{\partial v}{\partial x}  \right).
\end{align*}
Substituting \eqref{CR} and \eqref{u_x_to_u_y_on_C} into the right-hand side above, we obtain
\[
\frac{\de}{\de x} \operatorname{Re}\left( f_0(z(x)) \right) = -Y(x)\cdot (1+Y'(x)^2)\cdot \frac{\partial u}{\partial y}(x,Y(x)).
\]
Then, using $Y(x)>0$ and \eqref{Re(f')_partial_y<0_on_C}, we have
\begin{equation}
\frac{\de}{\de x} \operatorname{Re}\left( f_0(z(x)) \right) > 0 \quad\text{for any}\ x\in(\mu_0,\mu_1).
\label{742}
\end{equation}
By Equations \eqref{741}, \eqref{742}, and \eqref{Y_is_continuous_at_end_points}, we conclude that $\lambda \mapsto \operatorname{Re}(f_0(\tau_{\lambda}))$ is strictly increasing on $[0,k)$.
\end{proof}
		
At the end of this section, we establish the asymptotic formula for the linear form $S_n$ (see Definition \ref{def_S_n} and Lemma \ref{lem_linear_forms}) as $n \to +\infty$.
		
\begin{lemma}
\label{lem_asym_S_n}
Let $\tau_{k-2}$ be as defined in Lemma \ref{roots_of_f'} (with $\lambda = k-2$).
Let $f_0(z)$ be the function in \eqref{def_f_0}.
Then, as $n \to +\infty$, we have
\[
|S_n| = \exp\big( -\alpha n + o(n) \big) \cdot \big( |\cos(n\omega + \varphi)| + o(1)  \big),
\]
where
\begin{equation}
\label{def_alpha_omega}
\alpha = -\operatorname{Re}\left( f_0(\tau_{k-2}) \right), \quad \omega = \operatorname{Im}(f_0(\tau_{k-2})),
\end{equation}
and
\begin{equation}
\label{def_varphi}
\varphi = -\frac{1}{2}\arg \left(f''(\tau_{k-2})\right) + \arg\left(g(\tau_{k-2})\right).   
\end{equation}
\end{lemma}
		
\begin{proof}
By Lemmas \ref{lem_tilde_S_n} and \ref{lem_S=sum_J}, we have $S_n=n^{O(1)}\cdot\widetilde{S}_n$, where
\[
\widetilde{S}_n= \sum_{0 \leqslant l \leqslant k-2 \atop l \equiv k \pmod{2}} c_l \operatorname{Re}(J_{n,l}),\qquad c_{k-2}\neq0.
\]
By Lemmas \ref{f_0(z)} and \ref{Re(f_0)}, the quantities $|J_{n,l}|$ ($l \neq k-2$) are exponentially smaller than $|J_{n,k-2}|$ as $n \to +\infty$. Thus, we have
\[
\widetilde{S}_n= c_{k-2}\operatorname{Re}(J_{n,k-2}) + o(|J_{n,k-2}|).
\]
Note that $J_{n,k-2} = \exp((-\alpha+i\omega+o(1))n +i\varphi+o(1))$, where the first $o(1)$ is a real quantity. 
Since $c_{k-2}\neq0$, we obtain
\[
|S_n|=\exp\big(-\alpha n + o(n)\big)\cdot\big(|\cos(n\omega + \varphi)|+o(1)\big). \qedhere
\]
\end{proof}
		
\section{Proof of the main theorem}
\label{Sect_proof}
		
Throughout this section, we fix an integer $k \geqslant 2$ and take 
\[
r = \lfloor \log^2 q\rfloor.
\]
Assume that $q$ is sufficiently large that $r>2k$. 
The complex number $\tau_{k-2}$ is given by Lemma \ref{roots_of_f'}. 
The real numbers $\alpha,\omega, \varphi$ are provided by Lemma \ref{lem_asym_S_n}. 
The real number $\beta$ is defined by \eqref{def_beta}. 
Note that $\alpha,\omega,\varphi,\beta$ depend only on $k$ and $q$. 
		
\begin{lemma}
\label{lem_asym_tau}
As $q \to +\infty$, we have 
\[
|\tau_{k-2}-q| = \exp\left( -\frac{\log^2 q}{k} + O(\log q) \right),
\]
where the implicit constant depends only on $k$.
\end{lemma}
		
\begin{proof}
Fix any $\varepsilon_0 \in (0,1)$. 
On the circle $|z-q| =  \varepsilon_0$, we have
\[
\left| \frac{z-q}{z+r+q}  \right|^k \geqslant \left( \frac{\varepsilon_0}{2q+r+\varepsilon_0} \right)^k \quad\text{and}\quad \left| \frac{z}{z+r} \right|^{q+k} \leqslant \left( \frac{q+\varepsilon_0}{q+r-\varepsilon_0} \right)^{q+k}.
\]
Note that for any sufficiently large $q \geqslant q_0(\varepsilon_0,k)$, we have
\[
\left( \frac{\varepsilon_0}{2q+r+\varepsilon_0} \right)^k > \left( \frac{q+\varepsilon_0}{q+r-\varepsilon_0} \right)^{q+k}.
\]
It follows from Rouch{\'e}'s theorem that the function
\[
\left( \frac{z-q}{z+r+q}  \right)^k - \left( \frac{z}{z+r} \right)^{q+k} = \frac{P(z)}{(z+r+q)^k(z+r)^{q+k}}
\]
has exactly $k$ zeros (counted with multiplicity) inside the disk $|z-q|\leqslant \varepsilon_0$, where $P(z)$ is the polynomial defined by \eqref{def_P}. 
Then, by Lemma \ref{roots_of_f'} (2), we obtain $|\tau_{k-2}-q| \leqslant \varepsilon_0$. 
In other words, 
\begin{equation}\label{811}
\tau_{k-2} = q + o(1) \quad\text{as~} q \to +\infty.
\end{equation}
			
Now, by the definition of $\tau_{k-2}$, we have $\operatorname{Re}\left(f'(\tau_{k-2})\right) = 0$; that is,
\begin{equation}
\label{812}
k\log|q-\tau_{k-2}| = -(q+k)\log\left| 1+ \frac{r}{\tau_{k-2}} \right| + k\log|\tau_{k-2}+r+q|.
\end{equation}
Substituting $r= \lfloor \log^2 q \rfloor$ and \eqref{811} into the right-hand side of \eqref{812}, we obtain the desired estimate
\[
|\tau_{k-2}-q| = \exp\left( -\frac{\log^2 q}{k} + O(\log q) \right). \qedhere
\]
\end{proof}
		
\begin{lemma}
\label{lem_asym_alpha_beta}
We have
\[
\alpha \sim q\log^3 q \quad\text{and}\quad \beta \sim (\log 2) \cdot q\log^2 q  \quad \text{as~} q \to +\infty.
\]
\end{lemma}
		
\begin{proof}
By Equations \eqref{def_alpha_omega} and \eqref{def_f_0}, we have
\begin{align}
\alpha =&~ -\operatorname{Re} f_0(\tau_{k-2}) \notag\\
=&~ \operatorname{Re}\Big( 
r(q+k)\log(\tau_{k-2}+r) - k(r+q)\log(\tau_{k-2}+r+q) - kq\log(q-\tau_{k-2}) \Big) \notag\\
&~-rq\log(r) - 2kq\sum_{p \mid q \atop p \text{~prime}} \frac{\log p}{p-1}. \label{821}
\end{align}
Substituting $r=\lfloor \log^2 q\rfloor$ into \eqref{821} and using Lemma \ref{lem_asym_tau}, we obtain
\[
\alpha = q \log^3 q + O\left( q \log^2 q \cdot \log\log q \right) \quad \text{as~} q \to +\infty.
\]
			
On the other hand, by \eqref{def_beta} and $r=\lfloor \log^2 q\rfloor$ we have
\begin{align*}
\beta &= rq\log 2 + k \left(  (2q+r)\log\left(q+\frac{r}{2}\right) - r\log \frac{r}{2} + 2q\sum_{p \mid q \atop p \text{~prime}} \frac{\log p}{p-1}  \right) \\
&= (\log 2)\cdot q \log^2 q + O(q \log q) \quad\text{as~} q \to +\infty. 
\end{align*}
\end{proof}
		
\begin{lemma}
\label{lem_varphi}
As $q \to +\infty$, we have
\[
\varphi \equiv \begin{cases}
\frac{\pi}{k}+o(1)  &\text{if $k\geqslant 3$ is odd} \\
\frac{\pi}{2}+\frac{\pi}{k}+o(1)  &\text{if $k \geqslant 2$ is even}
\end{cases}
\pmod{\pi}.
\]
\end{lemma}
		
\begin{proof}
By the definition of $\tau_{k-2}$ (see Lemma \ref{roots_of_f'} with $\lambda=k-2$), we have $f'(\tau_{k-2})=(k-2)\pi i$.
Differentiating \eqref{def_f} and taking the imaginary part, we obtain
\begin{equation}
\label{831}
(q+k)\Big(\arg(\tau_{k-2}) -\arg(\tau_{k-2}+r)\Big) + k\Big(\arg(\tau_{k-2}+r+q) - \arg(q-\tau_{k-2})\Big) = (k-2)\pi.
\end{equation}
By Lemma \ref{lem_asym_tau}, we have
\begin{equation}\label{8880}
\tau_{k-2} = q+ o(1).
\end{equation}
Therefore, 
\begin{equation}\label{8881}
\arg(\tau_{k-2}),\ \arg(\tau_{k-2}+r),\  \arg(\tau_{k-2}+r+q)\ = \frac{o(1)}{q}.
\end{equation}
Substituting \eqref{8881} into \eqref{831}, we obtain
\begin{equation}\label{8882}
\arg(q-\tau_{k-2}) = -\frac{(k-2)\pi}{k} +o(1).
\end{equation}

Differentiating \eqref{def_f} twice and using \eqref{8880}, we obtain 
\begin{align*}
f''(\tau_{k-2}) &=  (q+k) \left( \frac{1}{\tau_{k-2}} -\frac{1}{\tau_{k-2}+r} \right) + k \left( \frac{1}{\tau_{k-2}+r+q} + \frac{1}{q-\tau_{k-2}} \right) \\
&= \frac{k}{q-\tau_{k-2}} + O\left( \frac{\log^2 q}{q} \right).
\end{align*}
Therefore, by \eqref{8880} and \eqref{8882}, we have
\begin{equation}\label{8883}
\arg\left( f''(\tau_{k-2}) \right) = \frac{(k-2)\pi}{k} +o(1).
\end{equation}
By \eqref{def_g} and \eqref{8881}, every factor other than $(q-\tau_{k-2})^{k/2}$ has argument tending to zero. 
Therefore, by \eqref{8882}, we have
\begin{equation}\label{8884}
\arg\left( g(\tau_{k-2}) \right) \equiv -\frac{(k-2)\pi}{2} +o(1) \pmod{2 \pi}.
\end{equation}
By \eqref{def_varphi}, \eqref{8883}, and \eqref{8884}, we obtain
\[
\varphi \equiv -\frac{1}{2}\arg\left( f''(\tau_{k-2})\right) + \arg\left( g(\tau_{k-2}) \right) \equiv -\frac{(k+1)(k-2)\pi}{2k} +o(1) \pmod{\pi},
\]
which completes the proof.
\end{proof}
		
\bigskip
		
Now, we prove our main theorem.
		
\begin{proof}[\textbf{Proof of Theorem \ref{main}}]
Fix an integer $k \geqslant 2$ and let $q$ be a sufficiently large positive integer such that $r = \lfloor \log^2 q \rfloor > 2k$. 
			
For any $n \in q!\mathbb{N}$, Lemmas \ref{lem_linear_forms} and \ref{lem_beta} imply that
\[ 
\widehat{S}_n := qd_{rqn}^{k} \cdot S_n = \widehat{\rho}_{n,0} +  \widehat{\rho}_{n,1}\delta_k\zeta(k) + \sum_{1 \leqslant a < q/2} \widehat{\rho}_{n,a/q}  \zeta^{-}\left( k, \frac{a}{q} \right) 
\]
is a linear combination of 
\[
1, \ \delta_k\zeta(k), \ \zeta^{-}\left( k,\frac{a}{q} \right) \quad (1 \leqslant a < q/2)
\]
with integer coefficients. 
Note that the coefficients $\widehat{\rho}_{n,1}$ and $\widehat{\rho}_{n,a/q}$ ($1 \leqslant a < q/2$) are divisible by $d_{rqn}^k$.
Moreover, we have
\[ 
\max \left\{ |\widehat{\rho}_{n,0}|, |\widehat{\rho}_{n,1}|, |\widehat{\rho}_{n,a/q}| ~\Big|~ 1 \leqslant a < q/2 \right\} \leqslant \exp\left( \widehat{\beta} n + o(n) \right) \quad \text{as~} n \to +\infty,  
\]
where $\widehat{\beta} = \beta+krq$. 
By Lemma \ref{lem_asym_S_n}, we have
\[
|\widehat{S}_n| = \exp\big( -\widehat{\alpha} n + o(n) \big) \cdot \big( |\cos(n\omega + \varphi)| + o(1)  \big),
\]
where $\widehat{\alpha}=\alpha -krq$. 
			
If $q$ is sufficiently large, then Lemma \ref{lem_varphi} implies that
\[
\varphi \notin \frac{\pi}{2} + \pi\mathbb{Z}.  
\]
Let 
\[ 
d = \dim_{\mathbb{Q}}\operatorname{Span}_{\mathbb{Q}}\left( \left\{ 1,\delta_k\zeta(k) \right\} \bigcup \left\{  \zeta^-\left( k, \frac{a}{q} \right) ~\Big|~ 1 \leqslant a < \frac{q}{2} \right\} \right).
\]
By Theorem \ref{thm_Nes_osc}, we obtain
\[
d \geqslant 1+\frac{\widehat{\alpha}+(d-1)krq}{\widehat{\beta}}.
\]
Thus,
\[
d \geqslant 1 + \frac{\widehat{\alpha}}{\widehat{\beta}-krq}.
\]
Then, Corollary \ref{lem_S_n_is_in_Q+V^odd} (3) implies that
\[
\dim_{\mathbb{Q}}V_k^{-}(q) \geqslant d-1 \geqslant \frac{\widehat{\alpha}}{\widehat{\beta}-krq}.
\]
		
Finally, by Lemma \ref{lem_asym_alpha_beta} we have
\[
\widehat{\alpha} \sim q\log^3 q\quad\text{and}\quad \widehat{\beta} -krq \sim (\log 2) \cdot q\log^2 q \quad\text{as~} q \to +\infty.
\]
Therefore,
\[ 
\dim_{\mathbb{Q}} V_k^{-}(q) \geqslant \left(\frac{1}{\log 2} - o(1)\right) \cdot \log q \quad\text{as~} q \to +\infty.
\]
The proof of Theorem \ref{main} is complete.
\end{proof}

\section*{Declaration of generative AI use}
This work was essentially completed in 2025 without the use of AI tools.
In the 2026 revision process, the authors used ChatGPT (GPT-6 Astra) to check the proof.
GPT-6 Astra streamlined the statement of Theorem 2.2 and identified a correctable error concerning Lemma \ref{lem_varphi} in a previous version.
The authors corrected that error in light of GPT-6 Astra’s suggestions.
The authors are responsible for the mathematical content of this paper.

\section*{Acknowledgments}
        
The authors thank the referees for their careful reading and helpful comments, which improved the clarity and
quality of the manuscript.
L. L. is supported by the Research Foundation for Scholars of Xiamen University (Grant No. X2450218). J. L. would like to thank Professors Liang Xiao, Qingchun Tian, and Binyong Xie for their invaluable support and assistance throughout this research. J. L. is supported by Shanghai Key Laboratory of Pure Mathematics and Mathematical Practice and also supported by MOE Key Laboratory of Mathematics and Engineering Applications.

\vspace*{3mm}
\begin{flushright}
\begin{minipage}{148mm}\sc\footnotesize
L.\,L.: School of Mathematical Sciences, Xiamen University, Fujian, China \\
{\it E-mail address}: \href{mailto:lilaimath@gmail.com}{{\tt lilaimath@gmail.com}} \vspace*{3mm}
\end{minipage}
\end{flushright}
		
\begin{flushright}
\begin{minipage}{148mm}\sc\footnotesize
J.\,L.: School of Mathematical Sciences, East China Normal University, 500 Dongchuan Road, Minhang District, Shanghai 200241, People's Republic of China \\
{\it E-mail address}: \href{mailto:jialimath001@pku.org.cn}{{\tt jialimath001@pku.org.cn}} \vspace*{3mm}
\end{minipage}
\end{flushright}
		
\end{document}